\documentclass{article}

\usepackage[a4paper]{geometry}
\usepackage{mathtools,amssymb,bm,amsthm}
\usepackage{stmaryrd}
\usepackage{xcolor}
\usepackage{graphicx}
\usepackage{tikz-cd}
\usepackage{tikz}
\usepackage[pdfencoding=auto]{hyperref}

\newtheorem{theorem}{Theorem}[section]
\newtheorem*{theorem*}{Theorem}
\newtheorem{lemma}[theorem]{Lemma}
\newtheorem{proposition}[theorem]{Proposition}
\newtheorem{corollaire}[theorem]{Corollary}
\newtheorem{conjecture}[theorem]{Conjecture}
\theoremstyle{remark}
  \newtheorem{remark}[theorem]{Remark}
  
\theoremstyle{definition}
  \newtheorem{definition}[theorem]{Definition}
  \newtheorem{claim}{Claim}
  \newtheorem{problem}[theorem]{Problem}
  \newtheorem{example}[theorem]{Example}

\newcommand{\symgroup}[1]{\mathfrak{S}_{#1}}
\newcommand{\hecke}[1]{\mathcal{H}_{#1}}
\newcommand{\quantumfunctor}[1]{\mathcal{P}^{#1}_q}
\newcommand{\polyfunctor}[1]{\mathcal{P}^{#1}}
\newcommand{\field}{\mathbb{K}}
\newcommand{\integer}{\mathbb{Z}}
\newcommand{\cible}{\mathcal{V}}
\newcommand{\source}[1]{\Gamma^{#1}_q \mathcal{V}}
\newcommand{\koszul}{\kappa}

\DeclareMathOperator{\Hom}{Hom}
\DeclareMathOperator{\modulo}{mod}
\DeclareMathOperator{\Id}{Id}
\DeclareMathOperator{\End}{End}

\DeclareMathOperator{\Ext}{Ext}
\DeclareMathOperator{\Image}{Im}
\DeclareMathOperator{\frobenius}{Fr}
\DeclareMathOperator{\cartier}{c}

\title{Ext-group in the category of quantum polynomial functors via the quantum Frobenius twist}
\author{Deturck Théo}
\date{October 2024}

\begin{document}

\maketitle

\begin{abstract}
    We study the effect of a quantum Frobenius twist on Ext-groups in the category of quantum polynomial functors. We use quantum versions of the de Rham and Koszul complexes, and compute their homologies. We use them to do several Ext-computations, and obtain a formula to compute Ext-groups between two functors obtained via the Frobenius, in characteristic zero or in big enough characteristic. Finally, we make some advancements toward a general formula in arbitrary characteristic.
\end{abstract}

\section{Introduction}

Let $S_q(n,d)$ denote the quantum Schur algebra of rank $n$ and degree $d$ as in \cite{martin1993schur, parshall1991quantum}. Thus $S_q(n,d)\text{-}\mathrm{mod}$ is equivalent to the category of finite dimensional polynomial representations of degree $d$ of the quantum general linear group $GL_q(n)$. This quantum Schur algebra is a deformation of the classical Schur algebra $S_1(n,d)$. When $q$ is a an $\ell$-th root of unity, this algebra is usually not semi-simple (even when the ground field has characteristic zero) and very little is known about $\Ext$ between its modules,  beyond the consequences of the fact that quantum Schur algebras are quasi-hereditary \cite[Theorem 7.5.1]{martin1993schur}.

The purpose of this article is to show that much information about $\Ext$ can be deduced from the knowledge of $\Ext$ between classical Schur algebras.

To be more specific, when $q$ is a an $\ell$-th root of unity ($\ell\ge 2$ odd), the classical Schur algebras and their quantum versions are related by a morphism of algebras 
$$\frobenius: S_q(n,d\ell)\to S_1(n,d)$$
called the quantum Frobenius twist. In particular every classical polynomial representation $M$ of degree $d$ gives by pullback along $\frobenius$ a twisted polynomial representation $M^{(1)_q}$ of degree $d\ell$. Such twisted representations play an important role in the quantum representation theory in type A, for instance every simple polynomial representation of $GL_q(n)$ decomposes \cite[Thm 9.4.1]{parshall1991quantum} as a tensor product:
$$L_\alpha\otimes L_\beta^{(1)_q}$$
where $L_\alpha$ is a simple representation of $GL_q(n)$ with $\ell$-restricted highest weight $\alpha$ and $L_\beta$ is a simple polynomial representation of $GL(n)$. The main purpose of this article is to investigate the homological properties of twisted representations. In particular, we are interested in the following problem. 
\begin{problem}
What is the relation between $\Ext^*_{S_q(n,d\ell)}(M^{(1)_q},N^{(1)_q})$ and $\Ext^*_{S_1(n,d)}(M,N)$?
\end{problem}
We tackle this problem by using the category $\quantumfunctor{}$ of quantum polynomial functors introduced in \cite{hong2017quantum}. This category is a quantum version of the category $\polyfunctor{}_1$ of strict polynomial functors introduced by Friedlander and Suslin in \cite{friedlander1997cohomology}, and the relation of $\quantumfunctor{}$ with quantum Schur algebras generalizes the relation of $\polyfunctor{}_1$ with classical Schur algebras. To be more specific, every quantum polynomial functor $F$ of degree $d$ gives rise to polynomial representations $F(n)$ of degree $d$ of $GL_q(n)$ for $n\ge 1$, and if $n\ge d$ there is a graded isomorphism 
$$\Ext^*_{\quantumfunctor{}}(F,G)\simeq \Ext^*_{S_q(n,d)}(F(n),G(n))\;.$$
The quantum Frobenius twist also exists in the context of polynomial functors, and we are naturally led to investigate the following problem.
\begin{problem}\label{problem2}
What is the relation between $\Ext^*_{\quantumfunctor{}}(F^{(1)_q},G^{(1)_q})$ and $\Ext^*_{\polyfunctor{}_1}(F,G)$?
\end{problem}
Problem \ref{problem2} looks similar to the classical problem of computing extensions between functors precomposed by the classical Frobenius twists, see e.g. \cite{friedlander1997cohomology}, \cite{touze2012troesch}, \cite{chalupnik2005extensions} or the survey \cite{touze2018cohomology}. 
As the first main result of the article, we completely solve this problem when the degrees $d$ of the functors $F$ and $G$ are invertible in the ground field $\field$ over which we are working, that is when $\field$ has characteristic zero, or when $\field$ has positive characteristic $p>d$. In these conditions, we obtain a closed formula computing $\Ext^*_{\quantumfunctor{}}(F^{(1)_q},G^{(1)_q})$ which is  similar to the formula known in the classical case. Namely, every strict polynomial functor $G$ canonically extends \cite[section 2.5]{touze2012troesch} to a graded functor $G^*$ defined on finite dimensional graded vector spaces. Therefore, for all graded vector space $E$ we have a well-defined graded functor $G^*(E\otimes I)$ such that $G^*(E\otimes I)(V)=G^*(E\otimes V)$ (see definition \ref{def:G(V*I)} for more details). We prove the following statement.
\begin{theorem}\label{thm-intro-1}
The Yoneda algebra $E_1=\Ext_{\quantumfunctor{}}(I^{(1)_q},I^{(1)_q})$ is the graded $\field$-algebra generated by a class $e$ of degree $2$,  with unique relation $e^{\ell}=0$:
$$E_1= \field[e]/\langle e^\ell\rangle\;.$$
In particular, as a graded vector space, $E_1^{j}=\field$ if $j=2i$ for $0\le i<\ell$, and $E_1^{j}=0$ otherwise.

If $F$ and $G$ are homogeneous strict polynomial functors of degree $d$, with $d!$ invertible in $\field$, there is an isomorphism, natural with respect to $F$ and $G$:
$$\Ext^k_{\quantumfunctor{}}(F^{(1)_q},G^{(1)_q})\simeq \Hom_{\polyfunctor{}_1}\left(F,G^k(E_1\otimes I)\right)\;.$$
\end{theorem}

The isomorphism given in theorem \ref{thm-intro-1} allows for easy concrete computations. The next example gives an instance of such a concrete computation, additional concrete computations are given in section \ref{sec:Extlarge}.
\begin{example}
Let $\field$ be a field of characteristic zero or of characteristic $p>d$. Then there is a graded isomorphism:
$$\Ext^*_{\quantumfunctor{}}(\Lambda^{d\,(1)_q},\Lambda^{d\,(1)_q}) \simeq S^d(E_1)\;.$$
\end{example}

The isomorphism of theorem \ref{thm-intro-1} relies on the explicit computation of $\Ext^*_{\quantumfunctor{}}(I^{\otimes d\,(1)_q},I^{\otimes d\,(1)_q})$ in theorem \ref{theo:ext(tenseur,tenseur)}, which is actually valid without any restriction on $\field$.

In order to compute the latter, one of the main computational tools (which may also be of interest for other problems than the ones tackled in this article) is the quantum De Rham complex. We compute its homology, namely we construct in theorem \ref{cartier} a Cartier isomorphism between the quantum Frobenius twist of the classical De Rham complex and the homology of the quantum De Rham complex. Our computation is new, and complements the results of \cite{wambst1993complexes} by giving an explicit example of a complex associated to a Hecke pair which is not exact, but whose homology is nonetheless computable. 

The expert reader will notice that the De Rham complexes are also the main computation tool in the classical setting \cite{friedlander1997cohomology}. However, contrarily to \cite{friedlander1997cohomology}, the calculations done in our proof of theorem \ref{thm-intro-1} avoid spectral sequences.

In the last part of the article we examine problem \ref{problem2} when the characteristic of the ground field $\field$ is small. The situation is significantly more complicated in this case. We do not completely solve problem \ref{problem2} in this context, however, we prove the following result, which is a complete answer to problem \ref{problem2} when $F=I^{(k-1)}$ is the $(k-1)$-th classical Frobenius twist for $k\ge 1$. In this case $F^{(1)_q}$ is denoted by $I^{(k)_q}$, and more generally the quantum twist of the classical twist $G^{(k-1)}$ of a functor $G$ is denoted by $G^{(k)_q}$. (See section \ref{sec:Extsmall}.)
\begin{theorem}\label{thm-intro-2}
The graded vector space $E_r=\Ext_{\quantumfunctor{}}^*(I^{(r)_q},I^{(r)_q})$ is given by
\[
E_r^{j}=
\begin{cases}\field & \text{if $j=2i$ for $0\le i<\ell p^{r-1}$,}\\
0 & \text{otherwise.}
\end{cases}
\]
If $G$ is a homogeneous strict polynomial functor of degree $p^{r}>0$, then there is a graded isomorphism:
$$\Ext^k_{\quantumfunctor{}}(I^{(r+s)_q},G^{(r)_q})\simeq \bigoplus_{p^{s}i+j=k}E_{r}^i\otimes \Ext^i_{\polyfunctor{}_1}\left(I^{(r)},G\right)\;.$$
\end{theorem}

The isomorphism given in theorem \ref{thm-intro-2} also allows for easy concrete computations as in the following example. See section \ref{sec:Extsmall} for additional explicit examples.
\begin{example}
For all $r\ge 1$ and all $s\ge 0$ there is a graded isomorphism
$$\Ext^j_{\quantumfunctor{}}(I^{(r+s)_q},S^{p^{r}\,(s)_q})\simeq 
\begin{cases}
\field & \text{if $j=2p^ri$ for $0\le i<p^{s-1}\ell$,}\\
0 & \text{otherwise.}
\end{cases}
$$
\end{example}

The reader will notice that in contrast to theorem \ref{thm-intro-1}, theorem \ref{thm-intro-2} does not give a description of $E_r$ as a graded algebra (for the Yoneda product). We actually almost compute this algebra structure in theorem \ref{presentation E_r}, but we failed to determine a few structure constants. We conjecture that $E_r$ is graded commutative and in that case the missing constants would be equal to one. 

The statement of theorem \ref{thm-intro-2} is very similar to the results of \cite[Thm 4.5]{friedlander1997cohomology} and \cite[Thm 4.16]{giordano2023additive}. To prove theorem \ref{thm-intro-2}, we adapt the methods of \cite{friedlander1997cohomology} in the quantum setting, based on our quantum De Rham complexes and our replacement of Pirashvili's lemma (lemma \ref{lem:generalized pirashvili}), and we formalize the computations of \cite{giordano2023additive} and observe that they work without knowing the complete description of the graded algebra $E_r$.

The statement of theorem \ref{thm-intro-1} and \ref{thm-intro-2} are special cases of the following general conjecture.

\begin{conjecture}\label{conj}
    For all strict polynomial functors $F,G$, there is a graded isomorphism
    \[
        \Ext^k_{\quantumfunctor{}}(F^{(1)_q},G^{(1)_q}) \simeq \bigoplus_{i+j=k} \Ext^i(F,G^j(E_1 \otimes I)).
    \]
\end{conjecture}

Further steps towards a complete solution of problem \ref{problem2}, involving quantum analogues of Troesch complexes in the spirit of \cite{touze2012troesch,touze2018cohomology} will appear in an article in preparation \cite{deturck2}.

\section{Quantum polynomial functors}

For the various properties of quantum polynomial functor, and the different conventions and notations that we use, see \cite{hong2017quantum}. Here, we only summarize some of the basic properties and definitions for the reader's convenience. We fix a field $\field$ and a non-zero scalar $q \in \field$.

\subsection{Homogeneous quantum polynomial functors and $q$-Schur algebras}

\begin{definition}
    The Hecke algebra $\hecke{d}$ is the $\field$-algebra generated by $T_1,...,T_{d-1}$ with relations
    \[
        \begin{array}{rclr}
            T_i T_j & = & T_j T_i & \mbox{if } |i-j|>1 \\
            T_i T_{i+1} T_i & = & T_{i+1} T_i T_{i+1} & \\
            (T_i-q)(T_i + q^{-1}) & = & 0 &
         \end{array}
    \]
    As a vector space, $\hecke{d}$ has a standard basis $T_w$ indexed by permutations $w \in \symgroup{d}$. When $q=1$, we find back the algebra of the symmetric group. For $d=0$ or $d=1$, we let $\hecke{d} = \field$.
\end{definition}
Let $V_n = \field^n$ with canonical basis $e_1,...,e_n$. We define a Yang-Baxter operator on $V_n$ by 
\[
    R_n(e_i \otimes e_j) = \left \{
    \begin{array}{ll}
        e_j \otimes e_i & \mbox{if } i < j, \\
        q e_j \otimes e_i & \mbox{if } i=j, \\
        (q-q^{-1}) e_i \otimes e_j + e_j \otimes e_i & \mbox{if } i>j.
    \end{array}
    \right .
\]
These operators satisfy the relation
\[
    (R_n - q)(R_n + q^{-1}) = 0.
\]
We use this operator to define an action of $\hecke{d}$ on $V_n^{\otimes d}$ by letting $T_i$ act as
\[
    \Id^{\otimes i-1} \otimes R_n \otimes \Id^{d-i-1}.
\]
\begin{definition}[see \cite{hong2017quantum}]
Let $\source{d}$ be the category whose objects are vector spaces $V_n^{\otimes d}$, for $n\ge 1$, and whose morphisms are the $\hecke{d}$-linear maps, let $\cible$ be the category of $\field$-vector spaces of finite dimension and $\field$-linear maps. The category $\quantumfunctor{d}$ of homogeneous quantum polynomial functors of degree $d$ is the category of $\field$ linear functors from $\source{d}$ to $\cible$.
\end{definition}

For $F \in \quantumfunctor{d}$, we denote $F(n) = F(V_n^{\otimes d})$. 
When $q=1$ the category $\polyfunctor{d}_1$ is equivalent to the category of strict polynomial functors of Friedlander and Suslin \cite{friedlander1997cohomology}. To avoid cumbersome notation and to emphasize the link with Schur algebras, we introduce the following definition.

\begin{definition}
    We define the rectangular $q$-Schur bimodule $S_q(n,m;d)$ as
    \[
        S_q(n,m;d) = \Hom_{\source{d}}(V_n^{\otimes d},V_m^{\otimes d})
    \]
    When $n=m$, the endomorphism algebra $S_q(n,n;d)$ is equal to the classical $q$-Schur algebra $S_q(n,d)$.
    The composition defines maps $S_q(r,m;d) \otimes S_q(n,r;d) \to S_q(n,m;d)$. In particular, $S_q(n,m;d)$ is an $S_q(m,m;d)-S_q(n,n;d)$-bimodule.
\end{definition}

For $F \in \quantumfunctor{d}$, $F(n)$ is naturally a $S_q(n,n;d)$-module of finite dimension. In fact, we have the following proposition.
\begin{proposition}[see {\cite[theorem 4.7]{hong2017quantum}}]
    The functor
    \[
    \begin{array}{rcl}
        \quantumfunctor{d} & \to & S_q(n,n;d)-mod \\
        F & \mapsto & F(n) 
    \end{array}
    \]
    defines an equivalence of categories when $n \geq d$.
\end{proposition}

\subsection{The braided monoidal structure and the duality}

\begin{definition}
The category of quantum polynomial functors is the direct sum
$$\quantumfunctor{} = \bigoplus_{d \geq 0} \quantumfunctor{d}\;.$$
Thus, every quantum polynomial functor is a finite direct sum of homogeneous quantum polynomial functors, and every morphism preserves the degree. 
\end{definition}

We now recall the braided monoidal structure and the duality on $\quantumfunctor{}$.
\begin{definition}
    An element $f \in S_q(n,m;d_1+d_2)$ is an $\hecke{d_1 + d_2}$-linear map. In particular, $f$ is an $\hecke{d_1} \times \hecke{d_2}$-linear map, if we identify $\hecke{d_1} \times \hecke{d_2}$ with the subalgebra of $\hecke{d_1+d_2}$ generated by $T_1,...,T_{d_1-1},T_{d_1+1},...,T_{d_1+d_2-1}$. But
    \[
        \begin{array}{rcl}
            \Hom_{\hecke{d_1} \times \hecke{d_2}}(V_n^{\otimes d_1 + d_2},V_m^{\otimes d_1 + d_2}) & \cong & \Hom_{\hecke{d_1}}(V_n^{\otimes d_1},V_m^{\otimes d_1}) \otimes \Hom_{\hecke{d_2}}(V_n^{\otimes d_2},V_m^{\otimes d_2}) \\
            & = & S_q(n,m;d_1) \otimes S_q(n,m;d_2)
        \end{array}
    \]
    Thus, we can decompose $f$ as a sum
    \[
        f = \sum f_1 \otimes f_2 \in S_q(n,m;d_1) \otimes S_q(n,m;d_2).
    \]
    This defines a map $\Delta : S_q(n,m;d_1+d_2) \to S_q(n,m;d_1) \otimes S_q(n,m;d_2)$. Now, let $F \in \quantumfunctor{d_1}$ and $G \in \quantumfunctor{d_2}$. We define $F \otimes G \in \quantumfunctor{d_1+d_2}$ by
    \[
        (F \otimes G)(n) = F(n) \otimes G(n), \quad (F \otimes G)(f) = \sum F(f_1) \otimes G(f_2)
    \]
    where $\Delta(f) = \sum f_1 \otimes f_2$ as above.
\end{definition}

We introduce some important quantum polynomial functors.
\begin{definition}
    \begin{itemize}
        \item We start by the unit for the monoidal structure. This is the constant functor $\bigotimes^0 \in \quantumfunctor{0}$ given by
        \[
            \begin{array}{cc}
                 \bigotimes^0(n) = \field, \quad \mbox{and} \quad \bigotimes^0(f) = f \ \ \mbox{for } f \in S_q(n,m;0) = \field \Id.
            \end{array}
        \]
        \item Let $I = \bigotimes^1 \in \quantumfunctor{1}$ be the identity functor. It is given by
        \[
            I(n) = V_n \quad \quad \mbox{ and } \quad \quad I(f) = f \quad \mbox{for } f \in S_q(n,m;1) = \Hom_{\field} (V_n, V_m).
        \]
        \item For $d \geq 1$, we let $\bigotimes^d = \underbrace{I \otimes \cdots \otimes I}_{d}$.
    \end{itemize}
    
\end{definition}

\begin{proposition}\label{definition of the braiding}
    The monoidal category $\quantumfunctor{}$ is braided. More precisely, for any $F,G \in \quantumfunctor{}$, there is an isomorphism $R = R_{F,G} : F \otimes G \to G \otimes F$. These isomorphisms are uniquely determined by the following properties :
    \begin{enumerate}
        \item $R_{\bigotimes^1,\bigotimes^1} : \bigotimes^1 \otimes \bigotimes^1 \to \bigotimes^1 \otimes \bigotimes^1$ is given by
        $
            R(v \otimes w) = R_n(v \otimes w)
        $
        for $v,w \in V_n$.
        \item For any $F \in \quantumfunctor{}$, the following diagram commutes:
        \begin{equation}\label{unit of the braiding}
            \begin{tikzcd}
                \bigotimes^0 \otimes F \arrow[rr,"R"] \arrow[dr,"\cong"] 
                && F \otimes \bigotimes^0 \;.\\
                & F \arrow[ur,"\cong"] & 
            \end{tikzcd}
        \end{equation}
    
        \item For any $F,G,H \in \quantumfunctor{}$, the following two diagrams commute
        \begin{equation}\label{(FG)H to H(FG)}
            \begin{tikzcd}
                (F \otimes G) \otimes H \arrow[dr,"1 \otimes R"] \arrow[rr,"R"] 
                & & H \otimes (F \otimes G) \;,\\
                & F \otimes H \otimes G \arrow[ur,"R \otimes 1"] &
            \end{tikzcd}
        \end{equation}
        \begin{equation}\label{F(GH) to (GH)F}
            \begin{tikzcd}
                F \otimes (G \otimes H) \arrow[dr,"R \otimes 1"] \arrow[rr,"R"] 
                & & (G \otimes H) \otimes F \;.\\
                & G \otimes F \otimes H \arrow[ur,"1 \otimes R"] &
            \end{tikzcd}
        \end{equation}
        \item For any pair of natural transformations $f : F \to F'$, $g : G \to G'$, the diagram commutes:
        \begin{equation}\label{natural transformation and braiding}
            \begin{tikzcd}[column sep = huge]
                F \otimes G \arrow[r,"R"] \arrow[d,"f \otimes g"] 
                & G \otimes F \arrow[d,"g \otimes f"] \;.\\
                F' \otimes G' \arrow[r,"R"]
                & G' \otimes F'
            \end{tikzcd}
        \end{equation}
    \end{enumerate}
\end{proposition}
This definition of the braiding coincides with the definition given in \cite[section 5]{hong2017quantum}.
\begin{proof}
    The uniqueness of the braiding satisfying these properties follows from the fact that any quantum polynomial functor is a subquotient of some direct sum of tensor power functors $\bigotimes^d$. More precisely, using these properties, we show \cite[Lemma 5.1]{hong2017quantum}. With \eqref{natural transformation and braiding}, we compute the braiding on direct sums of tensor power functors. Again with \eqref{natural transformation and braiding}, we show that this determine the braiding on each of its subquotient. For the existence, we just need to show that the braiding defined in \cite[section 5]{hong2017quantum} satisfies all the properties. The first one is a particular case of \cite[Lemma 5.1]{hong2017quantum}, while the second and the third are parts of the definition of a braiding. So we just need to show the last property. We will use the notations of \cite[section 5]{hong2017quantum}. By definition,
    \[
        R(f \otimes g)(v \otimes w) = \sum \sigma(f(v)_1,g(w)_1) g(w)_0 \otimes f(v)_0.
    \]
    Meanwhile
    \[
        (g \otimes f)R(v \otimes w) = \sum \sigma(v_1,w_1) g(w_0) \otimes f(v_0).
    \]
    But since $f$ and $g$ are natural transformations,
    \[
        \sum f(v)_0 \otimes f(v)_1 = \sum f(v_0) \otimes v_1, \quad \mbox{and} \quad \sum g(w)_0 \otimes g(w)_1 = \sum g(w_0) \otimes w_1.
    \]
    It follows that $R(f \otimes g) = (g \otimes f)R$.
\end{proof}

\begin{definition}
    We identify $V_n$ with its dual vector space $V_n^*$ by identifying the standard basis with its dual basis ($e_i \mapsto e_i^*$). With this identification, we denote by $\sigma(f)$ the transpose of $f\in S_q(n,m;d)$. This defines an anti-automorphism $\sigma : S_q(n,m;d) \to S_q(m,n;d)$.
    
    Now, if $F \in \quantumfunctor{d}$, we let
    \[
        F^\#(n) = (F(n))^*, \quad F^\#(f) = (F(\sigma(f)))^*
    \]
    This define a quantum polynomial functor $F^\# \in \quantumfunctor{d}$. Moreover, if $f : F \to G$ is a natural transformation, then $f^\# : G^\# \to F^\#$ given by $f^\#(n) = (f(n))^*$ is a natural transformation.
\end{definition}
This notion of duality works well with the monoidal structure.
\begin{proposition}[see {\cite[lemma 3.4, proposition 5.6]{hong2017quantum}}]
    For $F,G \in \quantumfunctor{}$,
    \[
        (F \otimes G)^\# = F^\# \otimes G^\#, \quad (R_{F,G})^\# = R_{G^\#,F^\#}\;.
    \]
\end{proposition}

\subsection{Weights}

Let us look in more details the $\hecke{d}$-module $V_n^{\otimes d}$. This module decomposes as a direct sum of cyclic $\hecke{d}$-modules \cite[proposition 11.5]{takeuchi2002short} :
\[
    V_n^{\otimes d} = \bigoplus_{\alpha \in \Omega(d,n)} V_\alpha
\]
where $\Omega(d,n)$ is the set of compositions of $d$ in $n$ parts (with possibly some $\alpha_i = 0$). Here, $V_\alpha$ is the $\hecke{d}$-submodule generated by
\[
    e^{\otimes \alpha} = \underbrace{e_1 \otimes \cdots \otimes e_1}_{\alpha_1} \otimes \cdots \otimes \underbrace{e_n \otimes \cdots \otimes e_n}_{\alpha_n}.
\]
This decomposition enables us to define several elements of $S_q(n,n;d) = \End_{\hecke{d}}(V_n^{\otimes d})$, given by the composition
\[
    \xi_\alpha =  V_n^{\otimes d} \twoheadrightarrow V_\alpha \hookrightarrow V_n^{\otimes d}
\]
The $\xi_\alpha$ form a complete set of orthogonal idempotents of $S_q(n,n;d)$. Hence, given any $S_q(n,n;d)$-module $M$, we can decompose $M$ into a direct sum
\[
    M = \bigoplus_{\alpha \in \Omega(d,n)} M_\alpha \quad \quad \mbox{where} \quad M_\alpha = \xi_\alpha M.
\]
The $\alpha$ such that $M_\alpha \neq 0$ are called the weights of $M$. In the same way, if $F \in \quantumfunctor{d}$, then for each $n \geq 1$, $F(n)$ is an $S_q(n,n;d)$-module, and hence we can define the weights of a quantum polynomial functor as the compositions $\alpha$ of $d$ such that, if $\alpha$ has $n$ part, $F(n)_\alpha \neq 0$. 
\begin{proposition}
    Let $F,G \in \quantumfunctor{}$, then for any composition $\alpha$ in $n$ parts,
    \[
        (F \otimes G)(n)_\alpha = \bigoplus_{\alpha^1 + \alpha^2 = \alpha} F(n)_{\alpha^1} \otimes G(n)_{\alpha^2} \quad \mbox{and} \quad F^\#(n)_\alpha = (F(n)_\alpha)^*
    \]
    where the direct sum is over all pairs of compositions $(\alpha^1,\alpha^2)$ in $n$ parts whose sum (coordinate wise) is $\alpha$, and $(F(n)_\alpha)^*$ is the dual vector space of $F(n)_\alpha$.
\end{proposition}
\begin{proof}
    By definition, as a $\hecke{d_1} \times \hecke{d_2}$-module,
    \[
        V_\alpha = \bigoplus_{\alpha^1 + \alpha^2 = \alpha} V_{\alpha^1} \otimes V_{\alpha^2}
    \]
    where the sum runs over all pairs $(\alpha^1,\alpha^2)$ of compositions with $\alpha^1 \in \Omega(d_1,n)$, $\alpha^2 \in \Omega(d_2,n)$ and $\alpha^1 + \alpha^2 = \alpha$. Hence,
    \[
        \Delta(\xi_\alpha) = \sum_{\alpha^1 + \alpha^2 = \alpha} \xi_{\alpha^1} \otimes \xi_{\alpha^2}
    \]
    where the sum runs over all pairs $(\alpha^1,\alpha^2)$ of compositions such that $\alpha^1 + \alpha^2 = \alpha$. The first equality follows. For the second, we can verify that
    $
        \sigma(\xi_\alpha) = \xi_\alpha
    $,
    whence the second equality.
\end{proof}

\section{The quantum symmetric and exterior algebra}

In this section, we introduce two important quantum polynomial functors, which are quotients of $\bigotimes^d$. To define them, consider the tensor algebra $T(n)$ of $V_n$, and consider the two-sided ideals $I_S(n)$ and $I_\Lambda(n)$ generated by the vectors
\[
    R_n(v_1 \otimes v_2) - q v_1 \otimes v_2 \quad \mbox{for } v_1,v_2 \in V_n
\]
for $I_S(n)$ and
\[
    R_n(v_1 \otimes v_2) + q^{-1} v_2 \otimes v_1 \quad \mbox{for } v_1,v_2 \in V_n
\]
for $I_\Lambda(n)$. These two ideals are graded.
\[
    I_S(n) = \bigoplus_{d \geq 0} I_S^d(n) \quad \mbox{and} \quad I_\Lambda(n) = \bigoplus_{d \geq 0} I_\Lambda^d(n).
\]
The functors $I_S^d$ and $I_\Lambda^d$ are subfunctors of $\bigotimes^d$, hence they define homogeneous quantum polynomial functors of degree $d$. Finally, we define the quantum symmetric and exterior power as the quotients $S^d_q = \bigotimes^d / I_S^d$ and $\Lambda^d_q = \bigotimes^d / I_\Lambda^d$.

The direct sums
\[
    S^*_q(n) = \bigoplus_{d \geq 0} S^d_q(n) \quad \mbox{and} \quad \Lambda^*_q(n) = \bigoplus_{d \geq 0} \Lambda^d_q(n)
\]
are graded algebras with $S^0_q(n) = \Lambda^0_q(n) = \field$ and $S^1_q(n) = \Lambda^1_q(n) = V_n$. They are generated in degree 1 by the standard basis $e_1,...,e_n$ of $V_n$ with relations
\[
    e_j e_i = q e_i e_j \quad \mbox{if } i<j
\]
for $S^*_q(n)$, and
\[
    e_j \wedge e_i = \left \{
    \begin{array}{cc}
        (-q^{-1}) e_i \wedge e_j & \mbox{if } i<j\;,  \\
        0 & \mbox{if } i=j\;,
    \end{array}
    \right .
\]
for $\Lambda^*_q(n)$. The products $\mu_S : S^{d_1}_q \otimes S^{d_2}_q \to S^{d_1 + d_2}_q$ and $\mu_ \Lambda : \Lambda^{d_1}_q \otimes \Lambda^{d_2}_q \to \Lambda^{d_1 + d_2}_q$ define two natural transformations. Moreover, let $R$ denote the braiding of $\quantumfunctor{}$, and for $F,G \in \quantumfunctor{}$, $v \in F(n)$, $w \in G(n)$, denote
\[
    R(v \otimes w) = \sum r(w) \otimes r(v).
\]
Then, if $a,b,c \in S_q(n)$ or $\in \Lambda_q(n)$,
\begin{equation}\label{braiding and product on the left}
    \sum r_1(b) r_2(c) \otimes r_2(r_1(a)) = \sum r(bc) \otimes r(a)
\end{equation}
where the indices denote the order of the application of the braiding. In other words, we first switch $a$ and $b$ via $R$, then $a$ and $c$.  Similarly,
\begin{equation}\label{braiding and product on the right}
    \sum r_2(r_1(c)) \otimes r_2(a)r_1(b) = \sum r(c) \otimes r(ab).
\end{equation}
\begin{remark} The properties \eqref{braiding and product on the left} and \eqref{braiding and product on the right} are part of the definition of a Yang-Baxter algebra, see \cite[section 4]{hashimoto1992quantum}.
\end{remark}
Moreover, for $a \in S^i_q(n)$, $b \in S^j_q(n)$,
\begin{equation}\label{commutativity for symmetric power}
    \sum r(b) r(a) = q^{ij} ab
\end{equation}
and for $a \in \Lambda^i_q(n)$, $b \in \Lambda^j_q(n)$,
\begin{equation}\label{commutativity for exterior power}
    \sum r(b) r(a) = (-q^{-1})^{ij} ab.
\end{equation}
Both relations can be proven using \cite[lemma 5.1]{hong2017quantum}.

Now, to construct the de Rham and the Koszul complexes, we will define coalgebra structure on both $S^*_q$ and $\Lambda^*_q$.
\begin{proposition}
    The map $\bigotimes^{i+j} \to \bigotimes^i \otimes \bigotimes^j$ induced by the action of the Hecke algebra element
    \[
        \sum_{\sigma \in \symgroup{}^{(i,j)}} q^{l(\sigma)} T_{\sigma}
    \]
    where $\symgroup{}^{(i,j)}$ is the set of $(i,j)$-shuffles, passes to the quotient to a map
    \[
        \Delta^{(i,j)} : S^{i+j}_q \to S^i_q \otimes S^j_q
    \]
    These maps $\Delta^{(i,j)}$ define a graded coproduct $\Delta : S^*_q \to S^*_q \otimes S^*_q$. This gives $S^*_q$ a structure of coalgebra (with counit $\epsilon$ the projection to $S^0_q = \field$).
\end{proposition}
\begin{proposition}
    The map $\bigotimes^{i+j} \to \bigotimes^i \otimes \bigotimes^j$ induced by the action of the Hecke algebra element
    \[
        \sum_{\sigma \in \symgroup{}^{(i,j)}} (-q^{-1})^{l(\sigma)} T_{\sigma}
    \]
    where $\symgroup{}^{(i,j)}$ is the set of $(i,j)$-shuffles, passes to the quotient to a map
    \[
        \Delta^{(i,j)} : \Lambda^{i+j}_q \to \Lambda^i_q \otimes \Lambda^j_q
    \]
    These maps $\Delta^{(i,j)}$ define a graded coproduct $\Delta : \Lambda^*_q \to \Lambda^*_q \otimes \Lambda^*_q$. This gives $\Lambda^*_q$ a structure of coalgebra (with counit $\epsilon$ the projection to $\Lambda^0_q = \field$).
\end{proposition}
For the proofs, see \cite[section 4]{hashimoto1992quantum}. This coproduct has properties similar to the properties \eqref{braiding and product on the left} and \eqref{braiding and product on the right} of the products. If we use the Sweedler notation,
\[
    \Delta(a) = \sum a_1 \otimes a_2,
\]
then
\begin{equation}\label{braiding and coproduct on the left}
    \sum r_2(r_1(b)) \otimes r_2(a_1) \otimes r_1(a_2) = \sum r(b) \otimes r(a)_1 \otimes r(a)_2,
\end{equation}
\begin{equation}\label{braiding and coproduct on the right}
    \sum r_1(b_1) \otimes r_2(b_2) \otimes r_2(r_1(a)) = \sum r(b)_1 \otimes r(b)_2 \otimes r(a).
\end{equation}
Finally, we also have a relation of the form "$\Delta(ab) = \Delta(a) \Delta(b)$",
\begin{equation}\label{product and coproduct}
    \sum (ab)_1 \otimes (ab)_2 = \sum (\pm q^{\pm 1})^{|a_2||b_1|} a_1 r(b_1) \otimes r(a_2) b_2
\end{equation}
where $|c| = i$ if $c \in S^i_q(n)$ or $\Lambda^i_q(n)$, and $\pm = +$ if $a,b \in S^*_q$, $\pm = -$ if $a,b \in \Lambda^*_q$.

The following proposition recapitulates the result of \cite{hong2017quantum} that we will use, see \cite[Theorem 4.7, Corollary 4.10 and Proposition 6.4]{hong2017quantum}.
\begin{proposition}\label{property of symmetric power}
    If $\alpha \in \Omega(d,n)$, let $S_q^{\alpha} = S_q^{\alpha_1} \otimes \cdots \otimes S_q^{\alpha_n} \in \quantumfunctor{d}$. Then for any $F \in \quantumfunctor{d}$,
    \[
        \Hom_{\quantumfunctor{}}(F,S_q^{\alpha}) \cong F^{\#}(n)_\alpha
    \]
    where $F^\#$ is the dual of $F$. Moreover the $S^\alpha_q$ where $\alpha$ run over all partitions of $d$ form an injective cogenerating family of the category $ \quantumfunctor{d}$.
\end{proposition}

We consider now the decomposition of $S^d_q$ and $\Lambda^d_q$ into weight spaces.
\begin{proposition}
    Let $\alpha \in \Omega(d,n)$. Then $S^d_q(n)_\alpha$ is a one dimensional vector space, spanned by the vector
    \[
        e^\alpha = e_1^{\alpha_1} e_2^{\alpha_2} \cdots e_n^{\alpha_n}
    \]
    The vector space $\Lambda^d_q(n)_\alpha$ is zero if $\alpha$ has at least one part greater than $2$, and if $\alpha$ has all its part equal to $0$ or $1$, then $\Lambda^d_q(n)_\alpha$ is one dimensional. More precisely, if $J = \{j_1,...,j_d\} \subseteq \{1,...,n\}$ has $d$ elements, let $\alpha_J \in \Omega(d,n)$ be the composition defined by
    \[
        (\alpha_J)_i = \left \{
        \begin{array}{cc}
             1 & \mbox{if } i \in J, \\
             0 & \mbox{if } i \not \in J.
        \end{array}
        \right .
    \]
    Then $\Lambda^d_q(n)_{\alpha_J}$ is the one dimensional vector space spanned by
    \[
        \overline{e_J} = e_{j_1} \wedge e_{j_2} \wedge \cdots \wedge e_{j_d}.
    \]
\end{proposition}
\begin{proof}
    By definition, $S^d_q(n)_\alpha$ is the image of $V_\alpha$ under the projection map $\mu : V_n^{\otimes d} \twoheadrightarrow S^d_q(n)$. But since $V_\alpha$ is a cyclic $\hecke{d}$-module generated by $e^{\otimes \alpha}$, and since $\mu(e^{\otimes \alpha} \cdot T_w) = q^{\ell(w)} e^\alpha$, $S^d_q(n)_\alpha$ is spanned by $e^\alpha$.

    A similar reasoning holds for $\Lambda^d_q$, but
    \[
        \underbrace{e_1 \wedge \cdots \wedge e_1}_{\alpha_1} \wedge \cdots \wedge \underbrace{e_n \wedge \cdots \wedge e_n}_{\alpha_n} = 0
    \]
    if one of the $\alpha_i$ is greater than $1$.
\end{proof}

\section{Quantum Frobenius twist}
In this section, we introduce quantum Frobenius twists and establish their main properties. We rely on the quantum Frobenius twist for Schur algebras defined in \cite{parshall1991quantum}.

Assume that $q$ is a primitive $\ell$th root of $1$, with $\ell>1$ an odd integer. Then, by \cite[theorem 11.7.1]{parshall1991quantum}, there is a natural algebra homomorphism
\[
    \frobenius : S_q(n,n;d\ell) \to S_1(n,n;d)
\]
such that for $\alpha \in \Omega(d\ell,n)$
\[
    \frobenius(\xi_\alpha) = \left \{
    \begin{array}{cc}
        \xi_{\beta} & \mbox{if } \alpha = \ell\beta, \\
        0 & \mbox{otherwise}.
    \end{array}
    \right .
\]
We want to use it to define more general morphisms $S_q(n,m;dl) \to S_1(n,m;d)$. To do that, take $N \geq n,m$. Then $V_n^{\otimes d}$ can be identified with
\[
    \bigoplus_{\substack{\alpha \in \Omega(d,N) \\ \alpha_{n+1} = \cdots = \alpha_N = 0}} V_\alpha \subseteq V_N^{\otimes d}
\]
and there is a similar identification for $V_m^{\otimes d}$. Thus, given $f \in S_q(n,m;d)$, we can associate to $f$ the following element of $S_q(N,N;d)$
\[
    \tilde{f} : V_N^{\otimes d} \twoheadrightarrow \left ( \bigoplus_{\substack{\alpha \in \Omega(d,N) \\ \alpha_{n+1} = \cdots = \alpha_N = 0}} V_\alpha \right ) \cong V_n^{\otimes d} \xrightarrow{f} V_m^{\otimes d} \cong \left ( \bigoplus_{\substack{\alpha \in \Omega(d,N) \\ \alpha_{m+1} = \cdots = \alpha_N = 0}} V_\alpha \right ) \hookrightarrow V_N^{\otimes d}.
\]
Let
\[
    \xi_{n} = \sum_{\substack{\alpha \in \Omega(d,N) \\ \alpha_{n+1} = \cdots = \alpha_N = 0}} \xi_\alpha, \quad \quad
    \xi_{m} = \sum_{\substack{\alpha \in \Omega(d,N) \\ \alpha_{m+1} = \cdots = \alpha_N = 0}} \xi_\alpha
\]
Then the image of $S_q(n,m;d)$ inside $S_q(N,N;d)$ is
\[
    \xi_m S_q(N,N;d) \xi_n \subseteq S_q(N,N;d)
\]
Now, we apply this in degree $d\ell$. The image by $\frobenius$ of $\xi_n$ is given by
\[
    \frobenius(\xi_n) = \sum_{\substack{\alpha \in \Omega(d,N) \\ \alpha_{n+1} = \cdots = \alpha_N = 0}} \xi_\alpha = \xi_n
\]
where the $\xi_n$ on the right hand side is in $S_1(N,N;d)$. Thus,
\[
    S_q(n,m;d\ell) \cong \xi_m S_q(N,N;d\ell) \xi_n \xrightarrow{\frobenius} \xi_m S_1(N,N;d) \xi_n \cong S_1(n,m;d)
\]
defines a homomorphism $\frobenius : S_q(n,m;d\ell) \to S_1(n,m;d)$. Moreover, since the $\xi_n$ are idempotent, these homomorphisms commute with all the composition maps. This yields a functor $\frobenius : \source{d\ell} \to \Gamma^d_1 \cible$ with 
    \[
        \frobenius(V_n^{\otimes d}) = (\field^n)^{\otimes d}
    \]
    such that for any composition $\alpha \in \Omega(d\ell,n)$,
    \[
        \frobenius(\xi_\alpha) = \left \{
    \begin{array}{cc}
        \xi_{\beta} & \mbox{if } \alpha = \ell\beta, \\
        0 & \mbox{otherwise}
    \end{array}
    \right . \quad \in S_1(n,n;d).
    \]
\begin{definition}
    Pullback along the functor $\frobenius : \source{d\ell} \to \Gamma^d_1 \cible$ defines a functor
    \[
        -^{(1)_q} : \polyfunctor{d}_1 \to \quantumfunctor{d\ell}, \quad F \mapsto F^{(1)_q}
    \]
    called the quantum Frobenius twist.
\end{definition}

\begin{proposition}\label{natural embedding}
    The maps
    \[
        \varphi(n) : S^{d(1)_q}(n) \to S^{d\ell}_q(n), \quad \varphi(n)(e^{\alpha}) = e^{\ell\alpha} 
    \]
    define an injective natural transformation $\varphi : S^{d(1)_q} \to S^{d\ell}_q$.
\end{proposition}
\begin{proof}
    Identify $S_q^{d\ell}(n)$ with $A_q(n,1)_{d\ell} = S_q(1,n;d\ell)^*$ (see \cite[section 2.2]{hong2017quantum}), and similarly identify $S^{d(1)_q}(n)$ with $A_1(n,1)_d = S_1(1,n;d)^*$. Then this map is the dual of the quantum Frobenius twist (by definition of the quantum Frobenius twist, see \cite[section 7.2 and theorem 11.7.1]{parshall1991quantum}) which define a natural transformation by compatibility with the products.
\end{proof}

\begin{proposition}\label{property of quantum Frobenius twist}
    The quantum Frobenius twist is fully faithful and exact. Moreover :
    \begin{itemize}
        \item For any $F,G \in \polyfunctor{}_1$, 
        \[
            (F \otimes G)^{(1)_q} = F^{(1)_q} \otimes G^{(1)_q}, \quad \mbox{and} \quad (R_{F,G})^{(1)_q} = R_{F^{(1)_q},G^{(1)_q}}\;. 
        \]
        \item For any $F \in \polyfunctor{}_1$,
        \[
            (F^\#)^{(1)_q} = (F^{(1)_q})^\#\;.
        \]
        \item For any $f \in \polyfunctor{d}_1$ and any $\alpha \in \Omega(d\ell,n)$,
        \[
            F^{(1)_q}(n)_\alpha = \left \{
            \begin{array}{cc}
                 F(n)_\beta & \mbox{if } \alpha = \ell \beta, \\
                 0 & \mbox{otherwise.} 
            \end{array}
            \right .
        \]
    \end{itemize}
    Furthermore, for any $F  \in \quantumfunctor{}$ and any $G \in \polyfunctor{}_1$, $R_{F,G^{(1)_q}}$ and $R_{G^{(1)_q},F}$ are switch maps, i.e are given by
    \[
        R(v \otimes w) = w \otimes v.
    \]
\end{proposition}
\begin{proof}
    The exactness simply follows from the fact that if $f : F \to G$ is a natural transformation, $f^{(1)_q}(n) = f(n)$. If we add the fact that the algebra morphisms $S_q(n,m;d) \to S_1(n,m;d)$ are surjective (this is an easy consequence of the definition, see \cite[theorem 11.7.1]{parshall1991quantum}), this also shows that it is fully faithful. The compatibility with the tensor product comes from the fact the transpose of the quantum Frobenius twist is a bialgebra morphism \cite[section 7.2]{parshall1991quantum}. For the duality, we verify that
    \begin{center}
        \begin{tikzcd}
            S_q(n,m;d\ell) \arrow[r,"\frobenius"] \arrow[d,"\sigma"]
            & S_1(n,m;d) \arrow[d,"\sigma"]\\
            S_q(m,n;d\ell) \arrow[r,"\frobenius"]
            & S_1(m,n;d)
        \end{tikzcd}
    \end{center}
    commutes. The statement on weight spaces is a direct consequence of the fact that
    \[
        \frobenius(\xi_\alpha) = \left \{
        \begin{array}{cc}
            \xi_{\beta} & \mbox{if } \alpha = l\beta, \\
            0 & \mbox{otherwise}
        \end{array}
        \right . \quad \in S_1(n,n;d).
    \]
    The difficult part is to prove the compatibility with the braiding. $(R_{F,G})^{(1)_q} = R_{F^{(1)_q},G^{(1)_q}}$ is a particular case of the claim that $R_{F,G^{(1)_q}}$ and $R_{G^{(1)_q},F}$ are switch maps for any $F  \in \quantumfunctor{}$ and any $G \in \polyfunctor{}_1$ (the switch map defines a braiding on $\polyfunctor{}_1$ satisfying all the properties asked in proposition \ref{definition of the braiding}, and hence is the braiding on $\polyfunctor{}_1$). We prove that $R_{F,G^{(1)_q}}$ is a switch map. First note that by naturality of the braiding (proposition \ref{definition of the braiding} \eqref{natural transformation and braiding}), the braiding of a quotient is the quotient of the braiding and the braiding of a subobject is the restriction of the braiding. We prove the result in two steps. 
    \begin{itemize}
        \item \textbf{Step 1 :} We prove the result for $F = \bigotimes^0$ or $\bigotimes^1$ and $G = \bigotimes^0$ or $\bigotimes^1$. In the case where at least one of the two functor is $\bigotimes^0$, then this follow directly from proposition \ref{definition of the braiding} \eqref{unit of the braiding} (note that $(\bigotimes^0)^{(1)_q} = \bigotimes^0$). Now suppose $F = \bigotimes^1 \in \quantumfunctor{}$ and $G = \bigotimes^1 \in \polyfunctor{}_1$. We use the injective natural transformation $\varphi : G^{(1)_q} = S^{1(1)_q} \to S^\ell_q$ of proposition \ref{natural embedding}. $S^\ell_q$ is a quotient of $\bigotimes^\ell$, and the image of $\varphi$ is spanned by the image $e_j^\ell$ of the $e_j^{\otimes \ell}$ by this quotient map. Using proposition \ref{definition of the braiding} \eqref{F(GH) to (GH)F}, we see that
        \[
            R(e_i \otimes e_j^{\otimes \ell}) = e_i \otimes e_j^{\otimes \ell} \cdot T_1 T_2 \cdots T_\ell
        \]
        We separate the case $i<j$, $i=j$ and $i>j$.
        \begin{itemize}
            \item If $i<j$,
            \[
                R(e_i \otimes e_j^{\otimes \ell}) = e_j^{\otimes \ell} \otimes e_i 
            \]
            and hence
            \[
                R(e_i \otimes e_j^\ell) = e_j^\ell \otimes e_i .
            \]
            \item If $i=j$,
            \[
                R(e_i \otimes e_j^{\otimes \ell}) = q^\ell e_j^{\otimes \ell} \otimes e_i = e_j^{\otimes \ell} \otimes e_i
            \]
            and hence
            \[
                R(e_i \otimes e_j^\ell) = e_j^\ell \otimes e_i .
            \]
            \item If $i>j$,
            \[
                R(e_i \otimes e_j^{\otimes \ell}) = e_j^{\otimes \ell} \otimes e_i + (q-q^{-1}) \sum_{r=0}^{\ell-1} q^{\ell-r-1} e_j^{\otimes r} \otimes e_i \otimes e_j^{\otimes \ell-r}
            \] 
            and hence
            \begin{align*}
                R(e_i \otimes e_j^\ell) &
                = e_j^\ell \otimes e_i + (q-q^{-1}) \sum_{r=0}^{l-1} q^{\ell-r-1} e_j^r  e_i e_j^{\ell-r-1} \otimes e_j \\
                & = e_j^\ell \otimes e_i + (q-q^{-1}) \underbrace{(1 + q^2 + \cdots + (q^2)^{\ell-1})}_{=0} e_\ell^{\ell-1}e_i \otimes e_j \\
                & = e_j^\ell \otimes e_i
            \end{align*}
            using the relation $e_i e_j = q e_j e_i$.
        \end{itemize}
        Thus, in each case,
        \[
            R(e_i \otimes e_j^\ell) = e_j^\ell \otimes e_i.
        \]
        This show that $R_{\bigotimes^1,S^l_q}$ restricted to the image of $\varphi$ is a switch map, and hence that $R_{F,G^{(1)_q}}$ is the switch map.
        \item \textbf{Step 2 :} We extend the result to all functors. It follows from Step 1 and proposition \ref{definition of the braiding} \eqref{(FG)H to H(FG)} and \eqref{F(GH) to (GH)F} that $R_{F,G^{(1)_q}}$ is the switch map for $F = \bigotimes^d \in \quantumfunctor{}$ and $G = \bigotimes^e \in \polyfunctor{}_1$. If $F$ and $G$ are direct sum of tensor power functors, then since the restriction of the braiding on each summand of $F \otimes G$ is the switch map, $R_{F,G^{(1)_q}}$ is also the switch map. Finally, since the braiding works well with subobjects and quotients, we conclude that $R_{F,G^{(1)_q}}$ is the switch map for any $F,G$ subquotients of direct sum of tensor power functors. But every functor is of this form by proposition \ref{property of symmetric power}. This proves the proposition.
    \end{itemize}
    \end{proof}
\begin{remark} As it was observed in \cite[section 8.2]{buciumas2019quantum}, we cannot iterate this Frobenius twist. But when $\field$ is a field of positive characteristic, we have a quantum version of the iterated classical twist $F^{(r)_q} = (F^{(r-1)})^{(1)_q}$, where $^{(r-1)}$ refers to the classical Frobenius twist of strict polynomial functors \cite{friedlander1997cohomology}. More details in section \ref{sec:Extsmall}.
\end{remark}

\section{The Koszul and de Rham complexes}

\subsection{Definition and basic properties}

\begin{definition}
    For $0 \leq i \leq d$, set
    \[
        \Omega^i_{d} = S^{d-i}_q \otimes \Lambda^i_q\;,\qquad \Omega=\bigoplus_{i,d\ge 0} \Omega^i_{d}\;,
    \]
    and
    \[
        K^i_d = \Omega^{d-i}_d = S^i_q \otimes \Lambda^{d-i}_q\;,\qquad K=\bigoplus_{i,d\ge 0} K^i_{d}\;.
    \]
    The integer $d$ is called the degree and $i$ the cohomological degree. 
    The braiding of $\quantumfunctor{}$ and the algebra structure of the symmetric and exterior powers induce an algebra structure on $\Omega$ and $K$, with
    \[
        (a \otimes b) \cdot (c \otimes d) = (\mu \otimes \mu)(a \otimes R(b \otimes c) \otimes d).
    \]
    Moreover, we define two maps $\delta: \Omega \to \Omega$, $\koszul : K \to K$ raising the cohomological degree by 1, by the compositions:
    \[
        \delta : \Omega^i_{d} = S^{d-i}_q \otimes \Lambda^i_q \xrightarrow{\Delta^{(d-i-1,1)} \otimes 1} S^{d-i-1}_q \otimes S^1_q \otimes \Lambda^i_q = S^{d-i-1}_q \otimes \Lambda^1_q \otimes \Lambda^i_q \xrightarrow{1 \otimes \mu} S^{d-i-1}_q \otimes \Lambda^{i+1}_q
    \]
    and
    \[
        \koszul : K^i_{d} = S^i_q \otimes \Lambda^{d-i}_q \xrightarrow{1 \otimes \Delta^{(1, d- i - 1)} } S^i_q \otimes \Lambda^1_q \otimes \Lambda^{d-i-1}_q = S^i_q \otimes S^1_q \otimes \Lambda^{d-i-1}_q \xrightarrow{\mu \otimes 1} S^{i+1}_q \otimes \Lambda^{d-i-1}_q\;.
    \]
    Moreover, $\delta^2 = \koszul^2 = 0$. The complex $(\Omega,\delta)$ is called the quantum de Rham complex and the complex $(K,\koszul)$ is called the Koszul complex. When $q=1$, these two complexes are the classical de Rham and Koszul complex \cite[section 4]{friedlander1997cohomology}.
\end{definition}
The two complexes are special cases of the complexes in \cite{wambst1993complexes}, given by the Hecke symmetry $-qR_n$ and $q^{-1}R_n$. The following lemma is proved in \cite[lemma 3.3]{wambst1993complexes}.
\begin{proposition}\label{koszul/de Rham compatibility}
    On $\Omega^i_d = K^{d-i}_d$,
    \[
        (-q^{-1})^2 \delta \koszul + \koszul \delta = ((-q^{-1})^2)^i (d)_{q^2} \Id.
    \]
    Here, for $n \in \mathbb{N}$, $(n)_{q^2} = 1 + q^2 + \cdots + (q^2)^{n-1}$.
\end{proposition}

The following proposition is not proved in \cite{wambst1993complexes}. It will be a key technical point for the proofs of proposition \ref{homology of the de Rham complex} and theorem \ref{cartier}. We call it the $q$-Leibniz formula.
\begin{proposition}\label{compatibility of the differential with product}
    For $a \in K^{i}_{d_1}$ and $b \in K^{j}_{d_2}$,
    \[
        \koszul(ab) =  q^j \koszul(a) b + (-q^{-1})^{d_1 - i} a \koszul(b) 
    \]
    For $a \in \Omega^{i}_{d_1}$ and $b \in \Omega^{j}_{d_2}$,
    \[
        \delta(ab) =  q^{d_2 - j} \delta(a) b + (-q^{-1})^i a \delta(b) 
    \]
\end{proposition}
\begin{proof}
    We will prove this result for $\delta$ (the proof for $\koszul$ is similar). To do it, we first introduce some notation. First, we let
    \[
        a = \sum a' \otimes a'' \quad \mbox{and} \quad b = \sum b' \otimes b''
    \]
    Here each $a'$ is an element of $S^{d_1-i}_q$, each $b'$ of $S^{d_2-j}_q$, and each $a''$ is an element of $\Lambda^i_q$ and each $b''$ of $\Lambda^j_q$. When $v \in S^*_q$, we denote
    \[
        \Delta^{(*-1,1)}(v) = \sum v_1 \otimes v_2 \in S^{*-1}_q \otimes S^1_q.
    \]
    Then
    \[
        \delta(a) = \sum (a')_1 \otimes (a')_2 a'', \quad \delta(b) = \sum (b')_1 \otimes (b')_2b''
    \]
    and
    \begin{align*}
        \delta(ab) 
        & = \delta \left ( \left (\sum a' \otimes a'' \right ) \left (\sum b' \otimes b'' \right ) \right )\\
        & = \sum \delta((a' r(b') \otimes r(a'')b'') \\
        & = \sum (a' r(b'))_1 \otimes (a' r(b'))_2 r(a'')b''.
    \end{align*}
    Now, the relation \eqref{product and coproduct} translates in our case into the relation
    \[
        \sum (a' r(b'))_1 \otimes (a' r(b'))_2
        = \left ( \sum a'r(b')_1 \otimes r(b')_2 \right )
        + q^{d_2 - j} \left ( \sum a'_1 r_2(r_1(b')) \otimes r_2(a'_2) \right )
    \]
    Hence
    \begin{equation}\label{big sum of delta product}
        \delta(ab) = \left ( \sum a'r(b')_1 \otimes r(b')_2 r(a'')b'' \right )
        + q^{d_2 - j} \left ( \sum a'_1 r_2(r_1(b')) \otimes r_2(a'_2) r_1(a'') b'' \right )
    \end{equation}
    On the one hand, since the relation \eqref{braiding and product on the right} implies in our case
    \[
        \sum r_2(r_1(b')) \otimes r_2(a'_2)r_1(a'') = \sum r(b') \otimes r(a'_2 a'')
    \]
    we have
    \[
        \sum a'_1 r_2(r_1(b')) \otimes r_2(a'_2)r_1(a'')b'' = \sum a'_1 r(b') \otimes r(a'_2 a'') b'' = \delta(a)b
    \]
    On the other hand, by the relation \eqref{braiding and coproduct on the right}, we have
    \[
        \sum a'r(b')_1 \otimes r(b')_2 r(a'') b'' = \sum a'r_1(b'_1) \otimes r_2(b'_2) r_2(r_1(a'')) b''
    \]
    and by \eqref{commutativity for exterior power},
    \[
        \sum r_2(b'_2) r_2(r_1(a'')) = (-q^{-1})^i \sum r(a'')b'_2.
    \]
    We conclude that
    \[
        \sum a'r(b')_1 \otimes r(b')_2 r(a'')b''
        = (-q^{-1})^i \sum a'r(b'_1) \otimes r(a'')b'_2 b''
        = (-q^{-1})^i a \delta(b).
    \]
    Finally, we have
    \[
        \delta(ab) = (-q^{-1})^i a \delta(b) + q^{d_2 - j} \delta(a)b.
    \]
\end{proof}

\subsection{Homology of the Koszul and the de Rham complexes}

\begin{proposition}
    The Koszul complex is acyclic.
\end{proposition}
\begin{proof}
    It is a direct consequence of the fact that $S_q(n)$ is a Koszul algebra for any $n$, and of the fact that its Koszul dual is $\Lambda_q(n)$ \cite[exemple 3.4.1 and exercice 3.4.19]{witherspoon2019hochschild}. 
\end{proof}

\begin{proposition}\label{homology of the de Rham complex}
    The de Rham complex is acyclic in polynomial degree not divisible by $\ell$. In polynomial degree divisible by $\ell$, there is an isomorphism of vector spaces
    \[
        H^i(\Omega_{d\ell})(n) \cong \bigoplus_{\alpha \in \Omega(d,n)} \Omega^i_{d\ell}(n)_{\ell \alpha}
    \]
    for all $i \geq 0$.
\end{proposition}
\begin{proof}
    The acyclicity in polynomial degree not divisible by $l$ is proved in the exact same way as in \cite[proposition 4.1]{wambst1993complexes}. By proposition \ref{koszul/de Rham compatibility},
    \[
        h_i : \Omega^i_d \to \Omega^{i-1}_d  , \quad h_i = (-q^{-1})^{2(i-1)} \frac{1}{(d)_{q^2}} \koszul
    \]
    define a homotopy between the identity and the zero map $\Omega_d \to \Omega_d$, since $(d)_{q^2} \neq 0$.
    To compute the homology in polynomial degree $d\ell$, we decompose the de Rham complex $\Omega_{d\ell}$ as a direct sum of complexes
    \[
        \Omega_\alpha : 0 \to S^{d\ell}_q(n)_\alpha \to (S^{d\ell-1}_q \otimes \Lambda^1_q)(n)_{\alpha} \to \cdots \to \Lambda^{d\ell}_q(n)_{\alpha} \to 0.
    \]
    We make the following claim.
    \begin{claim}
        Let $\alpha$ be a composition in $n$ parts, and let $k$ be such that $\alpha_k \neq 0$. Let $\alpha'$ be the composition in $n$ parts given by
        \[
            (\alpha')_i = \left \{
            \begin{array}{cc}
                 \alpha_i & \mbox{if } i \neq k, \\
                 0 & \mbox{if } i = k.
            \end{array}
            \right .
        \]
        Then there is a short exact sequence of complexes :
        \[
            0 \to \Omega_{\alpha'}[1] \to \Omega_{\alpha} \to \Omega_{\alpha'} \to 0
        \]
        where $C[1]$ is the complex defined by $C[1]^i = C^{i-1}$.
    \end{claim}
    We prove claim 1 after the end of this proof. Before this, we use this claim to compute the homology of $\Omega_{\alpha}$ for the $\alpha \in \Omega(d\ell,n)$ which are not divisible by $\ell$. For them, choose $k$ such that $\alpha_k$ is not divisible by $\ell$. Then the corresponding $\alpha'$ in the above claim is a composition of $d\ell - \alpha_k$, which is not divisible by $\ell$. Hence, $\Omega_{\alpha'}$ is acyclic, and the short exact sequence prove that it is also the case for $\Omega_{\alpha}$. Now, we compute the homology of $\Omega_{\alpha}$ for the $\alpha \in \Omega(d\ell,n)$ which are divisible by $\ell$. This follow from claim 2.
    \begin{claim}
        Let $\alpha$ be a composition in $n$ part divisible by $\ell$. Then all differential in $\Omega_{\alpha}$ are zero.
    \end{claim}
    \end{proof}
    
\begin{proof}[Proof of claim 1]
    We have 
    \begin{align*}
        \Omega^i_\alpha & = \bigoplus_{J \subseteq Supp(\alpha), \ |J| = i} S^{d-i}_q(n)_{\alpha - \alpha_J} \otimes \Lambda^{i}_q(n)_{\alpha_J} \\
        & = \left ( \bigoplus_{J \subseteq Supp(\alpha'), \ |J| = i} S^{d-i}_q(n)_{\alpha - \alpha_J} \otimes \Lambda^{i}_q(n)_{\alpha_J} \right ) \\
        & \qquad \qquad \oplus \left ( \bigoplus_{J \subseteq Supp(\alpha'), \ |J| = i-1} S^{d-i}_q(n)_{\alpha - \alpha_J - 1_k} \otimes \Lambda^{i}_q(n)_{\alpha_J + 1_k} \right)
    \end{align*}
    where $Supp(\alpha)$ is the set of $i$ such that $\alpha_i \neq 0$, and $1_k$ is the composition with every part equals zero, except the $k$-th part which is equal to $1$. By proposition \ref{compatibility of the differential with product}, the graded morphism $\iota$ given by the multiplication by $q^{\alpha_k-1}e_k^{\alpha_k-1} \otimes e_k$ is a morphism of complexes. The image of $\iota$ is the second part of the above direct sum. We show that $\iota$ is injective. First note that
    \[
        q^{\alpha_k-1} (e_{k}^{\alpha_k-1} \otimes e_k) = (1 \otimes e_k) (e_k \otimes 1)^{\alpha_k-1}\;.
    \]
    Since $a \in \Omega_{\alpha'}$ and $(\alpha')_k = 0$, we have $a(1 \otimes e_k) \neq 0$. Moreover, a direct computation shows that if $a \in \Omega$ is non-zero, then $a(e_k \otimes 1) \neq 0$. These two properties show the injectivity. Now, we need to show that the cokernel of $\iota$ is isomorphic to $\Omega_{\alpha'}$. To do that, we let
    \[
        \psi : \Omega_{\alpha'} \to \Omega_{\alpha} / \Image(\iota) \mbox{ be defined by } \psi(a) = q^{i\alpha_k} [a(e_k^{\alpha_k} \otimes 1)] \mbox{ if } a \in \Omega^i_{\alpha'}
    \]
    The graded map $\psi$ is a morphism of complexes, since
    \[
        \delta(a(e_k^{\alpha_k} \otimes 1)) = q^{\alpha_k} \delta(a) (e_k^{\alpha_k} \otimes 1) + (-q^{-1})^{i} (\alpha_k)_{q^2} \underbrace{a(e_k^{\alpha_k-1} \otimes e_k)}_{\in \Image(\iota)}
    \]
    and hence
    \[
        \overline{\delta}(\psi(a)) = q^{\alpha_k} q^{i\alpha_k} [d(a)(e_k^{\alpha_k} \otimes 1)] =  \psi(\delta(a))
    \]
    Moreover, $\psi$ is surjective, since $\Omega_{\alpha} / \Image(\iota)$ is spanned by the $[e^{\alpha-\alpha_J} \otimes e_J]$ with $J \subseteq Supp(\alpha')$ such that $|J|=i$, and 
    \[
        [(e^{\alpha'-\alpha_J} \otimes e_J)(e_k^{\alpha_k} \otimes 1)] = [e^{\alpha-\alpha_J} \otimes e_J]\;.
    \]
    Finally, a dimension argument (or a proof by direct computation) proves that $\psi$ is also injective, and the claim follows.
\end{proof}

    \begin{proof}[Proof of claim 2]
        \begin{itemize}
        \item The complex $\Omega_{(\ell)}$ is concentrated in degrees $0$ and $1$, hence there is only one differential to compute. This can be done in a straightforward manner, and we find that this differential is zero.
        \item Next, we use \ref{compatibility of the differential with product} and the fact that the product $\Omega_{(\ell)} \otimes \Omega_{(d\ell)} \to \Omega_{((d+1)\ell)}$ is surjective to prove by induction that it is the case for all $\Omega_{(d\ell)}$.
        \item Finally, we use again \ref{compatibility of the differential with product} and the fact that the product
        \[
            \Omega_{(\alpha_1)} \otimes \Omega_{(\alpha_2)} \otimes \cdots \otimes \Omega_{(\alpha_n)} \to \Omega_{\alpha}
        \]
        (where the $i$-th part of the tensor product correspond to the $e_i$) is surjective to prove that it the claim also hold for $\Omega_{\alpha}$.
    \end{itemize}
    \end{proof}
\begin{theorem}\label{cartier}
    There is an isomorphism of graded algebras
    \[
        \cartier(n) : \Omega^{(1)_q}(n) \to H^{\bullet}(\Omega(n))
    \]
    given on generators by
    \[
        e_i \otimes 1 \mapsto [e_i^\ell \otimes 1] \in H^0(\Omega_\ell(n)), \quad 1 \otimes e_i \mapsto [e_i^{\ell-1} \otimes e_i] \in H^1(\Omega_\ell(n))
    \]
    Moreover, this isomorphism defines functor isomorphisms
    \[
        \cartier^i_d : \Omega_d^{i(1)_q} \to H^i(\Omega_{d\ell})
    \]
\end{theorem}
Our notation $c(n)$ is inspired by the fact that $c(n)$ is a quantum version of the classical Cartier isomorphism \cite{cartier1957nouvelle}.
\begin{proof}
    On the one hand, the algebra $\Omega^{(1)_q}(n)$ is equal to the classical algebra $S(V_n) \otimes \Lambda(V_n)$, hence is generated by $e_i \otimes 1$, $1 \otimes e_i$ for $i=1,...,n$, with relations
    \[
        \begin{array}{rcl}
             (e_i \otimes 1)(e_j \otimes 1) & = & (e_j \otimes 1)(e_i \otimes 1), \\
             (e_i \otimes 1)(1 \otimes e_j) & = & (1 \otimes e_j)(e_i \otimes 1), \\
             (1 \otimes e_i)(1 \otimes e_j) & = & -(1 \otimes e_j)(1 \otimes e_i), \\
             (1 \otimes e_i)(1 \otimes e_i) & = & 0.
        \end{array}
    \]
    On the other hand, the algebra $\Omega(n)$ is equal to the quantum algebra $S_q(V_n) \otimes \Lambda_q(V_n)$. Hence, using the relations in the quantum symmetric and quantum exterior algebras, and part 1 and 2 of proposition \ref{definition of the braiding}, we see that $\Omega(n)$ is generated by $e_i \otimes 1$, $1 \otimes e_i$ for $i=1,...,n$, with relations
    \[
        \begin{array}{rcll}
            (e_i \otimes 1)(e_j \otimes 1) & = & q(e_j \otimes 1)(e_i \otimes 1) & \mbox{if } i>j, \\
            (1 \otimes e_j)(e_i \otimes 1) & = & (e_i \otimes 1)(1 \otimes e_j) & \mbox{if } i>j, \\ 
            (1 \otimes e_i)(e_i \otimes 1) & = & q(e_i \otimes 1)(1 \otimes e_i), & \\ 
            (1 \otimes e_j)(e_i \otimes 1) & = & (q-q^{-1})(e_j \otimes 1)(1 \otimes e_i) + (e_i \otimes 1)(1 \otimes e_j) & \mbox{if } i<j, \\ 
            (1 \otimes e_i)(1 \otimes e_j) & = & (-q^{-1})(1 \otimes e_j)(1 \otimes e_i) & \mbox{if } i>j, \\
            (1 \otimes e_i)(1 \otimes e_i) & = & 0. &
        \end{array}
    \]
    Furthermore, we have
    \[
        \begin{array}{rcl}
             (e_i^\ell \otimes 1) & = & (e_i \otimes 1)^\ell, \\
             (e_i^{\ell-1} \otimes e_i) & = & (e_i \otimes 1)^{\ell-1} (1 \otimes e_i).
        \end{array}
    \]
    Using the relations in $\Omega(n)$ and this last remark, we prove by direct computation that
    \[
        \begin{array}{rcl}
            (e_i^\ell \otimes 1)(e_j^\ell \otimes 1) & = & (e_j^\ell \otimes 1)(e_i^\ell \otimes 1), \\
            (e_j^{\ell-1} \otimes e_j)(e_i^\ell \otimes 1) & = & (e_i^\ell \otimes 1)(e_j^{\ell-1} \otimes e_j), \\
            (e_i^{\ell-1} \otimes e_i)(e_j^{\ell-1} \otimes e_j) & = & - (e_j^{\ell-1} \otimes e_j)(e_i^{\ell-1} \otimes e_i), \\
            (e_i^{\ell-1} \otimes e_i)(e_i^{\ell-1} \otimes e_i) & = & 0.
        \end{array}
    \]
    Let $1_i$ be the composition given by
    \[
        1_i = (\underbrace{0,...,0}_{i-1},1,\underbrace{0,...,0}_{l-i}).
    \]
    Since $e_i^\ell \otimes 1 \in \Omega^0_\ell(n)_{l \cdot 1_i}$ and $e_i^{\ell-1} \otimes e_i \in \Omega^1_\ell(n)_{l \cdot 1_i}$, they are cycles by proposition \ref{homology of the de Rham complex}. 
    Thus, $\cartier(n)$ is a well-defined morphism of algebras. To prove that it is an isomorphism, recall that for $\alpha \in \Omega(d,n)$,
    \[
        (\Omega^i_{d\ell})(n)_{\ell \alpha} = \bigoplus_{J \subseteq Supp(\alpha)} S^{d\ell-i}_q(n)_{\ell\alpha - \alpha_J} \otimes \Lambda^i_q(n)_{\alpha_J}
    \]
    and hence has for basis the family of vector $e^{\ell\alpha - \alpha_J} \otimes e_J$, for $J \subseteq Supp(\alpha)$. Since
    \[
        e^{\ell\alpha - \alpha_J} \otimes e_J = \prod_{i=1}^n \left \{ 
        \begin{array}{cc}
             (e_i^\ell \otimes 1)^{\alpha_i} & \mbox{if } i \not \in J \\
             (e_i^\ell \otimes 1)^{\alpha_i - 1} (e_i^{\ell-1} \otimes e_i) & \mbox{if } i \in J
        \end{array}
        \right .
    \]
    we conclude that the map sends the basis $e^{\alpha - \alpha_J} \otimes e_J$ to the basis $e^{\ell\alpha - \alpha_J} \otimes e_J$, and hence is an isomorphism. To prove the naturality, since we have (from what we have proven above), the commutative diagram (in each $n \geq 1$)
    \begin{center}
        \begin{tikzcd}
            (\Omega_1^{0 (1)_q})^{\otimes d-i} \otimes (\Omega_1^{1 (1)_q})^{\otimes i}(n)
            \arrow[r, twoheadrightarrow, "mult"] \arrow[d,"\cong"] 
            & \Omega^{i (1)_q}_d(n) \arrow[d,"\cong"] \\
            (H^0(\Omega_\ell))^{\otimes d-i} \otimes (H^1(\Omega_\ell))^{\otimes i}(n)
            \arrow[r, twoheadrightarrow, "mult"]
            & H^i(\Omega_{d\ell})(n)
        \end{tikzcd}
    \end{center}
    where the vertical arrows are given by the Cartier isomorphism, and the horizontal arrows are surjective, it suffices to prove that the maps
    \[
        \begin{array}{rcl}
            \Omega^{0 (1)_q}_1 & \to & H^0(\Omega^\ell) \\
            e_i \otimes 1 & \mapsto & [e_i^\ell \otimes 1]
        \end{array}, \quad \mbox{and} \quad
        \begin{array}{rcl}
            \Omega^{1 (1)_q}_1 & \to & H^1(\Omega^\ell) \\
            1 \otimes e_i & \mapsto & [e_i^{\ell-1} \otimes e_i]
        \end{array}
    \]
    are natural transformations. To see that the first one is a natural transformation, note that the composition
    \[
        \begin{array}{ccccc}
             \Omega^{0 (1)_q}_1 & \xrightarrow{\cong} & H^0(\Omega_\ell) & \hookrightarrow & S^\ell_q \\
             e_i \otimes 1 & \mapsto & [e_i^\ell \otimes 1] & \mapsto & e_i^\ell
        \end{array}
    \]
    is a natural transformation. Since $H^0(\Omega_\ell) \hookrightarrow S^\ell_q$ is an injective natural transformation, this shows $\Omega^{0(1)_q} \to H^0(\Omega_\ell)$ is also a natural transformation.

    For the second one, let
    \begin{align*}
        Z & = \ker(\delta : \Omega^1_\ell \to \Omega^2_\ell) \\
        B & = \Image(\delta : \Omega^0_\ell \to \Omega^1_\ell)
    \end{align*}
    so that $H^1(\Omega_\ell) = Z/B$. The Koszul differential $\koszul : \Omega^1_\ell \to \Omega^0_\ell$ restricts to a natural transformation $\koszul : Z \to \Omega^0_\ell$ 
    . If $x \in B(n)$, then by definition $x = \delta(y)$ for some $y \in \Omega^0_\ell(n)$, and hence by proposition \ref{koszul/de Rham compatibility},
    \[
        \koszul(x) = \koszul \delta(y) = -(-q^{-1})^2 \delta \koszul(y) = 0
    \]
    since $(\ell)_{q^2} = 0$ and $\koszul(y) = 0$ since $y \in \Omega^0_\ell$. Thus, $\koszul$ induces a natural transformation $\overline{\koszul} : H^1(\Omega_\ell) = Z/B \to \Omega^0_\ell$. This map is given by
    \[
        \overline{\koszul}([e_i^{\ell-1} \otimes e_i]) = [e_i^\ell \otimes 1]
    \]
    and hence is an isomorphism $\overline{\koszul} : H^1(\Omega_\ell) \xrightarrow{\cong} H^0(\Omega_\ell)$. Moreover, the map $\cartier_1^1$ is simply the composition :
    \[
        \Omega^{1(1)_q } = \Lambda^{1(1)_q} \cong S^{1(1)_q} = \Omega^{0(1)_q} \xrightarrow{\cartier(1)^0} H^0(\Omega_\ell) \xrightarrow{\overline{\koszul}^{-1}} H^1(\Omega_\ell)
    \]
    and thus is also a natural transformation.
\end{proof}

\subsection{The Koszul kernel subcomplex}

\begin{definition}
    For all $d,i \geq 0$, we define the subfunctor $N^i_d$ of $\Omega^i_d$ by
    \[
        N^i_d = \ker(\koszul : \Omega^i_d \to \Omega^{i-1}_d)
    \]
\end{definition}
When $d$ is divisible by $\ell$, they form a subcomplex of $\Omega^{\bullet}_d$ by proposition \ref{compatibility of the differential with product}. It is used in the $\Ext$-group computation of theorem \ref{theo:Ext(fro,fro)}, in the same way that the classical Koszul kernel subcomplex is used in \cite[theorem 4.5]{friedlander1997cohomology}. The following proposition gives the link these two complexes.

\begin{proposition}\label{Homologie noyau de Koszul}
    For all $i$ and $d$, we have an isomorphism :
    \[
        N^{i(1)_q}_d \cong H^i(N_{d\ell}^\bullet)
    \]
    where the $N$ on the right hand side is the complex defined in the same way as in the classical case. This isomorphism is induced by the isomorphism $\Omega^{i (1)_q}_d \to H^i(\Omega^{\bullet}_{d\ell})$ of theorem \ref{homology of the de Rham complex}.
\end{proposition}
\begin{proof}
    We have again 
    \[
        H^i(N_{d\ell}^\bullet)(n) = \bigoplus_{\alpha \in \Omega(d,n)} (N_{d\ell}^i(n))_{\ell\alpha}
    \]
    as in the proof of theorem \ref{homology of the de Rham complex}. By proposition \ref{koszul/de Rham compatibility}, the koszul differential $\kappa$ induces maps $\overline{\kappa} : H^i(\Omega_{d\ell}) \to H^{i-1}(\Omega_{d\ell})$. Then, a simple computation shows that
    \begin{center}
        \begin{tikzpicture}[xscale = 3.5, yscale = 2]
            \node (M1) at (0,1) {$\Omega^{i (1)_q}_d$};
            \node (M2) at (1,1) {$\Omega^{i-1 (1)_q}_d$};
            \node (M3) at (0,0) {$H^i(\Omega^{\bullet}_{dl})$};
            \node (M4) at (1,0) {$H^{i-1}(\Omega^{\bullet}_{dl})$};

            \draw[->] (M1) -- (M2) node[midway,above] {$\kappa^{(1)}$};
            \draw[->] (M2) -- (M4) node[midway,right] {$c(n)^i$};
            \draw[->] (M1) -- (M3) node[midway,right] {$c(n)^{i-1}$};
            \draw[->] (M3) -- (M4) node[midway,above] {$\overline{\kappa}$};
        \end{tikzpicture}
    \end{center}
    commutes up to a (non-zero) multiplicative scalar. Since the kernel of $\overline{\kappa} : H^i(\Omega_{d\ell}) \to H^{i-1}(\Omega_{d\ell})$ is $H^i(N^\bullet_{d\ell})$, we deduce that the Cartier isomorphism sends $N^{i(1)_q}_d$ to $H^i(N_{d\ell}^\bullet)$. Moreover, using the acyclicity of the Koszul complex, for any $\beta \in \Omega(d\ell,n)$, we show
    \[
        \dim (N^i_{d\ell}(n)_{\beta}) = \dim(\Omega_{d\ell}^{i-1}(n)_\beta) - \dim(N^{i-1}_{d\ell}(n)_\beta)
    \]
    and similarly in the classical case, for any $\alpha \in \Omega(d,n)$
    \[
        \dim (N^{i(1)_q}_d(n)_{\alpha}) = \dim(\Omega^{i-1(1)_q}_d(n)_\alpha) - \dim(N^{i-1(1)_q}_d(n)_\alpha).
    \]
    Moreover, $\dim(N^0_{d\ell}(n)_{\beta}) = \dim(N^{0(1)_q}_d(n)_{\alpha}) = 1$, and knowing the weight spaces of $\Omega^i_{d\ell}$ and $\Omega^{i(1)_q}_d$, we show that
    \[
        \dim(\Omega_{dl}^{i}(n)_{\ell\alpha}) = \dim(\Omega^{i(1)_q}_d(n)_\alpha)
    \]
    for any $\alpha \in \Omega(d,n)$.
    Hence, by induction, 
    \[
        \dim(N^i_{d\ell}(n)_{\ell\alpha}) = \dim(N^{i(1)_q}_d(n)_{\alpha})
    \]
    for any $\alpha \in \Omega(d,n)$. This shows that $H^i(N_{d\ell})$ and $N^{i(1)_q}_d(n)$ have the same dimension. Hence, the restriction of the Cartier isomorphism to $N^{i(1)_q}_d$ defines an isomorphism from $N^{i(1)_q}_d$ to $H^i(N_{d\ell})$.
\end{proof}

\section{Application to $\Ext$ computations}\label{sec:Extlarge}

In this section, we will prove theorem \ref{thm-intro-1}. This theorem will follow from the computation of the $\Ext$-group $\Ext^*_{\quantumfunctor{}}(\bigotimes^{d(1)_q},\bigotimes^{d(1)_q})$. This computation will use in a crucial way the Koszul kernel complex $N^*_\ell$, whose homology is given by
\[
    H^*(N^*_\ell) \cong \left \{
    \begin{array}{cl}
        I^{(1)_q} & \mbox{if } *=0, \\
        0 & \mbox{otherwise.}
    \end{array}
    \right.
\]
We first prove a version of the Pirashvili lemma where the additivity condition is replaced by a condition on the weights.

\begin{lemma}\label{lem:generalized pirashvili}
    Let $A \in \quantumfunctor{}$ be a quantum polynomial functor. Assume that the weights of $A$ are compositions with at most $k$ non-zero parts. Let $F_1,...,F_{k+1} \in \quantumfunctor{}$ be homogeneous quantum polynomial functors of degree $\geq 1$. Then
    \[
        \Ext^*_{\quantumfunctor{}}(A,F_1 \otimes \cdots \otimes F_{k+1}) = 0.
    \]
\end{lemma}
\begin{proof}
    We start by proving that
    \[
        \Hom_{\quantumfunctor{}}(A,J_1 \otimes \cdots \otimes J_{k+1}) = 0 \quad (*)
    \]
    for any injective polynomial functors $J_1,...,J_{k+1}$ homogeneous of degree $\geq 1$. Since the $S^{\mu}_q$ form a cogenerating family of injective, we just need to prove $(*)$ for $J_i = S^{\alpha^i}_q$, with $\alpha^i \in \Omega(n_i,d_i)$ for some $n_i,d_i \geq 1$. In this case,
    \[
        \Hom_{\quantumfunctor{}}(A,S^{\alpha^1}_q \otimes \cdots \otimes S^{\alpha^{k+1}}_q) = 
        \Hom_{\quantumfunctor{}}(A,S^{\alpha^1 . \alpha^2. \cdots . \alpha^{k+1}}) =
        A^{\#}(n)_{\alpha^1 . \alpha^2. \cdots . \alpha^{k+1}}
    \]
    where $n = n_1 + \cdots + n_{k+1}$ and $\alpha.\beta$ denotes the concatenation product of composition. But since $d_1,...,d_{k+1} \geq 1$, each $\alpha^i$ has at least one non zero part, and hence the concatenation $\alpha^1 . \alpha^2. \cdots . \alpha^{k+1}$ has at least $k+1$ non-zero parts. Since moreover $A, A^\#$ have the same weight, our hypothesis implies
    \[
        A^{\#}(n)_{\alpha^1 . \alpha^2. \cdots . \alpha^{k+1}} = 0
    \]
    which proves $(*)$. Now we prove the $\Ext$ vanishing of the lemma. Take injective coresolutions $J_i^*$ of $F_i$. Then $J_1^* \otimes \cdots \otimes J_{k+1}^*$ is an injective coresolution of $F_1 \otimes \cdots \otimes F_{k+1}$, and hence we can compute $\Ext^*_{\quantumfunctor{}}(A,F_1 \otimes \cdots \otimes F_{k+1})$ as the homology of a complex $C$ with 
    \[
    C^s = \Hom_{\quantumfunctor{}}(A, \bigoplus_{s_1 + \cdots + s_{k+1} = s} J_1^{s_1} \otimes \cdots \otimes J_{k+1}^{s_{k+1}}).
    \]
    But $(*)$ proves that each $C^s$ is zero, whence the result.
\end{proof}
We start by computing $\Ext^*_{\quantumfunctor{}}(I^{(1)_q},I^{(1)_q})$. To perform this computation, we need the following lemma :
\begin{lemma}\label{lem:Ext(twist,kernel)}
    For $i=0,...,\ell-1$,
    \[
        \Ext^*_{\quantumfunctor{}}(I^{(1)_q},N^i_\ell) = \left \{ \begin{array}{cl}
            \field & \mbox{if } *=i, \\
            0 & \mbox{otherwise.}
        \end{array} \right .
    \]
\end{lemma}
\begin{proof}
    We start by $N^0_\ell = S^\ell_q$. Since it is injective, we only need to compute $\Ext^0_{\quantumfunctor{}}(I^{(1)_q},S^\ell_q) = \Hom_{\quantumfunctor{}}(I^{(1)_q},S^\ell_q)$. By proposition \ref{property of symmetric power},
    \[
        \Hom_{\quantumfunctor{}}(I^{(1)_q},S^\ell_q) \cong I^{(1)_q}(1)_{(\ell)} \cong \field.
    \]
    This proves the case $i=0$. Now, we have short exact sequences
    \[ 
    0 \to N^{i+1}_\ell \to S^{\ell-i-1}_q \otimes \Lambda^{i+1}_q \to N^i_\ell \to 0
    \]
    for $i=0,..., \ell - 2$. But for such $i$, $\Ext_{\quantumfunctor{}}^*(I^{(1)_q}, S^{\ell-i-1}_q \otimes \Lambda^{i+1}_q) = 0$ by Pirashvili lemma \ref{lem:generalized pirashvili}. Hence, the long exact sequence in homology shows that
    \[
        \Ext^*_{\quantumfunctor{}}(I^{(1)_q},N^{i+1}_\ell) \cong \Ext^{*-1}_{\quantumfunctor{}}(I^{(1)_q},N^{i}_\ell).
    \]
    We conclude the proof by induction.
\end{proof}

\begin{theorem}\label{theo:Ext(fro,fro)}
    As a graded $\field$-algebra, $\Ext^*_{\quantumfunctor{}}(I^{(1)_q},I^{(1)_q})$ is generated by an element $e$ in degree $2$, with only relation $e^\ell=0$.
\end{theorem}
\begin{proof}
    Since
    \[
        0 \to I^{(1)_q} \to N_0^l \to N^1_\ell \to \cdots \to N^{\ell-1}_\ell \to 0
    \]
    is an exact sequence, it splits into $\ell-1$ short exact sequences
    \[
        0 \to F_i \to N^i_\ell \to F_{i+1} \to 0
    \]
    where $F_i$ is the kernel of $N^i_\ell \to N^{i+1}_\ell$ (and hence $F_{\ell-1} = N_\ell^{\ell-1}$ and $F_0 = I^{(1)_q}$). Using this, we prove by induction that
    \[
        \Ext_{\quantumfunctor{}}^*(I^{(1)_q},F_i) = \left \{
        \begin{array}{cl}
            \field & \mbox{if } * = 2t-i \mbox{ for some } i \leq t \leq \ell-1, \\
            0 & \mbox{otherwise.}
        \end{array}
        \right .
    \]
    The induction goes as follows. First, since $F_{\ell-1} = N^{\ell-1}_\ell$, the case $i=\ell-1$ follows from lemma \ref{lem:ext(tenseur,kernel)}. Then, use the long exact sequence associated to the above short exact sequences to deduce the case $i$ from the case $i+1$, for $i=\ell-2,\ell-3,...,0$. Now, since 
    \[
        H^*(\Omega_\ell) = \left \{
        \begin{array}{cl}
            I^{(1)_q} & \mbox{if } *=0,1\ ,\\
            0 & \mbox{otherwise.}
        \end{array}
        \right .
    \]
    the de Rham complex splits into short exact sequences
    \[
        0 \to G_i \to \Omega^i_\ell \to G_{i+1} \to 0
    \]
    for $i=1,...,\ell-1$, where $G_i$ is the kernel of $\Omega^i_\ell \to \Omega^{i+1}_\ell$, and two more short exact sequences
    \[
        0 \to I^{(1)_q} \to S^\ell_q \to G \to 0 \ \ (*_1), \qquad 0 \to G \to G_1 \to I^{(1)_q} \to 0 \ \ (*_2),
    \]
    where $G$ is the image of $\Omega^0_\ell \to \Omega^{1}_\ell$. We can prove by induction that for $i=1,...,\ell$,
    \[
        \Ext^*_{\quantumfunctor{}}(I^{(1)_q}, G_i) = \left \{
        \begin{array}{cl}
            \field & \mbox{if } * = 2\ell-1-i, \\
            0 & \mbox{otherwise.}
        \end{array}
        \right .
    \]
    This follows from the fact that $\Ext^*_{\quantumfunctor{}}(I^{(1)_q},\Omega^i_\ell) = 0$ for $i=1,...,\ell-1$, which can be deduced from Pirashvili lemma \ref{lem:generalized pirashvili}. In particular
    \[
        \Ext^*_{\quantumfunctor{}}(I^{(1)_q}, G_1) = \left \{
        \begin{array}{cl}
            \field & \mbox{if } * = 2\ell-2, \\
            0 & \mbox{otherwise.}
        \end{array}
        \right .
    \]
    Using the long exact sequence in cohomology induced by $(*_1)$, we obtain that
    \[
        \Ext^*_{\quantumfunctor{}}(I^{(1)_q},G) = \left \{
        \begin{array}{cl}
            \field & \mbox{if } * = 2t - 1 \mbox{ for some } 1 \leq t \leq \ell-1, \\
            0 & \mbox{otherwise.}
        \end{array}
        \right .
    \]
    and moreover the connecting morphism $\delta_1 : \Ext^*_{\quantumfunctor{}}(I^{(1)_q},G) \to \Ext^{*+1}_{\quantumfunctor{}}(I^{(1)_q},I^{(1)_q})$ is an isomorphism. Now, using the long exact sequence induced by $(*_2)$, we prove that the connecting morphism $\delta_2 : \Ext^*_{\quantumfunctor{}}(I^{(1)_q},I^{(1)_q}) \to \Ext^{*+1}_{\quantumfunctor{}}(I^{(1)_q},G)$ is an isomorphism for $* \neq 2\ell-2$. Let $e$ be any non-zero element of $\Ext^2_{\quantumfunctor{}}(I^{(1)_q},I^{(1)_q})$. Then, since $\delta_1 \circ \delta_2$ is a morphism of right $\Ext^*_{\quantumfunctor{}}(I^{(1)_q},I^{(1)_q})$-modules (see \cite[Chapter 3, theorem 9.1]{maclane2012homology}), for all $* \geq 0$, the following diagram
    \begin{center}
        \begin{tikzcd}
            \Ext^*_{\quantumfunctor{}}(I^{(1)_q},I^{(1)_q}) \arrow[r,"\delta_1 \circ \delta_2"] \arrow[d,"\cdot e"]
            & \Ext^{*+2}_{\quantumfunctor{}}(I^{(1)_q},I^{(1)_q}) \arrow[d,"\cdot e"] \\
            \Ext^{*+2}_{\quantumfunctor{}}(I^{(1)_q},I^{(1)_q}) \arrow[r,"\delta_1 \circ \delta_2"] 
            & \Ext^{*+4}_{\quantumfunctor{}}(I^{(1)_q},I^{(1)_q})
        \end{tikzcd}
    \end{center}
    commute. Since for $*=0$, the left map is an isomorphism, and since the top and bottom map are isomorphisms for $* \leq 2\ell-6$, we conclude by induction that for $* \leq 2\ell-6$, the right map is also an isomorphism, whence the result.
\end{proof}

In the rest of this section, we will use in a crucial way the cup product. Let us recall the necessary facts about it to fix notations.
\begin{definition}\label{def:cup}
    For all quantum polynomial functors $F_1,F_2,G_1,G_2$, there is a graded map
    \[
        \begin{array}{rcl}
            \Ext_{\quantumfunctor{}}(F_1,G_1) \otimes \Ext_{\quantumfunctor{}}(F_2,G_2) & \to & \Ext_{\quantumfunctor{}}(F_1 \otimes F_2, G_1 \otimes G_2) \\
            c_1 \otimes c_2 & \mapsto & c_1 \cup c_2
        \end{array}
    \]
    such that :
    \begin{itemize}
        \item In degree $0$, the cup product is given by the usual tensor product of morphisms.
        \item Let $F_i,G_i$, $i=1,2,3$, be quantum polynomial functors, and $c_i \in \Ext_{\quantumfunctor{}}(F_i,G_i)$. Then
        \[
            c_1 \cup (c_2 \cup c_3) = (c_1 \cup c_2) \cup c_3.
        \]
        \item Let $F_i,G_i,H_i$, $i=1,2$, be quantum polynomial functors, and $\phi_i : G_i \to H_i$, $i=1,2$, be morphisms. Then for any $c_i \in \Ext_{\quantumfunctor{}}(F_i,G_i)$, $i=1,2$, we have
        \[
            (\phi_1 \cdot c_1) \cup (\phi_2 \cdot c_2) = (\phi_1 \otimes \phi_2) \cdot (c_1 \cup c_2)
        \]
        where the $\cdot$ denote the Yoneda product.
        \item If we fix $F_1,F_2$ and $G_1$, then the cup product define a morphism of $\delta$-functors from $\Ext_{\quantumfunctor{}}(F_1,G_1) \otimes \Ext_{\quantumfunctor{}}(F_2,-)$ to $\Ext_{\quantumfunctor{}}(F_1 \otimes F_2,G_1 \otimes -)$. Similarly, if we fix $G_2$ instead of $G_1$, the cup product also give a morphism of $\delta$-functors.
    \end{itemize}
\end{definition}
    Now, since the Frobenius is fully faithful,
    \[
        \Hom_{\quantumfunctor{}}(\mbox{$\bigotimes^{d(1)_q},\bigotimes^{d(1)_q}$}) \cong \Hom_{\polyfunctor{}_1}(\mbox{$\bigotimes^d,\bigotimes^d$}) \cong \field \symgroup{d}.
    \]
    With this isomorphism, we identify $\Hom_{\quantumfunctor{}}(\mbox{$\bigotimes^{d(1)_q},\bigotimes^{d(1)_q}$})$ with $\field \symgroup{d}$. Using the cup product and the Yoneda product, we define for all $F_1,...,F_d \in \quantumfunctor{}$ a map
    \[
        \Psi(F_1,...,F_d) : \left \{
        \begin{array}{rcl}
            \left ( \bigotimes_{j=1}^d \Ext^{s_j}_{\quantumfunctor{}}(I^{(1)_q}, F_j) \right ) \otimes \field \symgroup{d} & \to & \Ext^s_{\quantumfunctor{}}(\bigotimes^{d(1)_q}, F_1 \otimes \cdots \otimes F_d) \\
            (c_1 \otimes \cdots \otimes c_d) \otimes \sigma & \mapsto & (c_1 \cup \cdots \cup  c_d) \cdot \sigma
        \end{array} \right .
    \]   
Note that, letting $F_i$ be a variable and the $F_j$ for $j \neq i$ be fixed, $\Psi(F_1,...,F_d)$ is a morphism between $\delta$-functors from $\quantumfunctor{}$ to $\mod-\field\symgroup{d}$.

A first application is the following lemma :
\begin{lemma}\label{lem:ext(tenseur,kernel)}
    For any composition $\alpha$ in $d$ parts $<\ell$, the map $\Psi(N_\ell^{\alpha_1},...,N_\ell^{\alpha_d})$ is an isomorphism of right $\field \symgroup{d}$-modules. Hence, if we let $N_\ell^{\alpha} = N_\ell^{\alpha_1} \otimes \cdots \otimes N_\ell^{\alpha_d}$, we have
    \[
        \Ext_{\quantumfunctor{}}^t(\mbox{$\bigotimes^{d(1)_q}$},N_\ell^{\alpha}) = \left \{
        \begin{array}{cc}
            \field \symgroup{d} & \mbox{if } t=|\alpha|, \\
            0 & \mbox{otherwise.}
        \end{array}
        \right .
    \]
\end{lemma}
\begin{proof}
    By acyclicity of the Koszul complex, we have short exact sequences
    \[
        0 \to N^i_\ell \to \Omega^i_\ell \to N^{i-1}_\ell \to 0.
    \]
    Since $\Omega^i_\ell = S^{\ell-i}_q \otimes \Lambda^i_q$, the Pirashvili lemma (lemma \ref{lem:generalized pirashvili}) implies that, if $i \neq 0,\ell$,
    \[
        \Ext^*_{\quantumfunctor{}}(I^{(1)_q}, \Omega^i_\ell) = \Ext_{\quantumfunctor{}}^*(\mbox{$\bigotimes^{d(1)_q}$}, N_\ell^{\alpha_1} \otimes \cdots \otimes \Omega_\ell^i \otimes \cdots \otimes N^{\alpha_d}_\ell) = 0.
    \]
    Thus, if we consider the $\delta$-functor of definition \ref{def:cup}, the connecting morphisms induced by this short exact sequence are isomorphisms. Thus, we only need to prove the lemma for $\alpha = (0,...,0)$. Since $N^0_\ell = S^\ell_q$ is injective, we only need to see what happens in degree $0$. In this case, since as a $\field$ vector space
    \[
        \begin{array}{rcll}
            \Hom_{\quantumfunctor{}}(\bigotimes^{d(1)_q}, (S^\ell_q)^{\otimes d}) & \cong & \bigotimes^{d(1)_q}(d)_{(\ell,...,\ell)} & \mbox{by proposition \ref{property of symmetric power},} \\
            & = & \bigotimes^{d}(d)_{(1,...,1)} & \mbox{by proposition \ref{property of quantum Frobenius twist},} \\
            & \cong  & \Hom_{\polyfunctor{}_1}(\bigotimes^{d},\bigotimes^{d}) & \mbox{by proposition \ref{property of symmetric power},} \\
        \end{array}
    \]
    the injective map 
    \[
        \begin{array}{rcl}
            \Hom_{\quantumfunctor{}}(I^{(1)_q},N_0^\ell)^{\otimes d} \otimes \field \symgroup{d} & \to & \Hom_{\quantumfunctor{}}(\bigotimes^{d(1)_q}, (N_0^\ell)^{\otimes d}) \\
            \varphi \otimes \cdots \otimes \varphi \otimes \sigma & \mapsto & \varphi^{\otimes d} \circ \sigma
        \end{array}
    \]
     where $\varphi$ is the injective map of proposition \ref{natural embedding}, is an isomorphism (recall that $\Hom_{\quantumfunctor{}}(I^{(1)_q},S^\ell_q) = \field \varphi$). This proves the lemma.
\end{proof}

We can now compute $\Ext^*_{\quantumfunctor{}}(\bigotimes^{d(1)_q},\bigotimes^{d(1)_q})$.
\begin{theorem}\label{theo:ext(tenseur,tenseur)}
    The map
    \[
        \begin{array}{rcl}
            \Ext^*_{\quantumfunctor{}}(I^{(1)_q},I^{(1)_q})^{\otimes d} \otimes \field \symgroup{d} &  \to & \Ext^*_{\quantumfunctor{}}(\bigotimes^{d(1)_q},\bigotimes^{d(1)_q}) \\
            c_1 \otimes \cdots \otimes c_d \otimes \sigma & \mapsto & (c_1 \cup \cdots \cup c_d) \cdot \sigma 
        \end{array}
    \]
    is an isomorphism. 
\end{theorem}
\begin{proof}
    As a consequence of lemma \ref{lem:ext(tenseur,kernel)}, the map given by the cup product
    \[
        \bigoplus_{s+t=*} \bigoplus_{\substack{\gamma \in \Omega(s,d) \\ \alpha \in \Omega(t,d)}} \left( \bigotimes_{j=1}^d \Ext^{\gamma_j}_{\quantumfunctor{}}(I^{(1)_q}, N^{\alpha_j}_\ell) \right ) \otimes \field \symgroup{d} \to \bigoplus_{s+t=*} \Ext_{\quantumfunctor{}}^{s}(\mbox{$\bigotimes^{d(1)_q}$}, (N^{\otimes d}_\ell)^t)
    \]
    is an isomorphism. Since moreover the homology of the complex $N^*_\ell$ is concentrated in degree $0$ and in this degree is $I^{(1)_q}$, a simple bicomplex argument shows that the left-hand side is isomorphic to
    \[
        \bigoplus_{\gamma \in \Omega(*,d)} \left( \bigotimes_{j=1}^d \Ext^{\gamma_j}_{\quantumfunctor{}}(I^{(1)_q}, I^{(1)_q}) \right ) \otimes \field \symgroup{d} = \Ext^*_{\quantumfunctor{}}(I^{(1)_q},I^{(1)_q})^{\otimes d}  \otimes \field \symgroup{d} \ .
    \]
    Similarly, the right-hand side is isomorphic to $\Ext^*_{\quantumfunctor{}}(\bigotimes^{d(1)_q},\bigotimes^{d(1)_q})$. A direct verification then shows that the composition of the isomorphisms is equal to the desired map.
\end{proof}
We can deduce from this theorem the algebra structure of $\Ext^*_{\quantumfunctor{}}(\bigotimes^{d(1)_q},\bigotimes^{d(1)_q})$.
\begin{corollaire}\label{cor:product}
    For any composition $\alpha$ in $d$ parts, we define
    \[
        e^{\alpha} = e^{\alpha_1} \cup \cdots \cup e^{\alpha_d} \in \Ext_{\quantumfunctor{}}^{2|\alpha|}(\mbox{$\bigotimes^{d(1)_q}$},\mbox{$\bigotimes^{d(1)_q}$})
    \]
    using the cup product. 
    For any composition $\alpha,\beta$ in $d$ parts and any $\sigma,\omega \in \symgroup{d}$, we have
    \[
        (e^{\alpha} \cdot \sigma) \cdot (e^{\beta} \cdot \omega) = e^{\alpha+\sigma(\beta)} \cdot \sigma \omega,
    \]
    where $\sigma(\beta) = (\beta_{\sigma^{-1}(1)},...,\beta_{\sigma^{-1}(d)})$.
\end{corollaire}
\begin{proof}
    It suffices to prove the result for $\sigma = \tau_i = (i,i+1)$ a basic transposition. Let $P_*$ be a projective resolution of $I^{(1)_q}$. Then $P^{\otimes d}$ is a projective resolution of $\bigotimes^{d(1)_q}$. We extend $\tau_i : \bigotimes^{d(1)_q} \to \bigotimes^{d(1)_q}$ to a chain map $\underline{\tau_i} : P^{\otimes d} \to P^{\otimes d}$ by letting its restriction to $P_{s_1} \otimes \cdots \otimes P_{s_d}$ be
    \[
         (-1)^{s_i s_{i+1}}(1^{\otimes i-1} \otimes R \otimes 1^{\otimes d-i-1}) : P_{s_1} \otimes \cdots \otimes P_{s_d} \to P_{s_1} \otimes \cdots \otimes P_{s_{i-1}} \otimes P_{s_{i+1}} \otimes P_{s_i} \otimes P_{s_{i+2}} \otimes \cdots \otimes P_{s_d}. 
    \]
    Now, if $f_j : P_{s_j} \to I^{(1)_q}$ are cocycles for $j=1,...,d$, then we have
    \begin{align*}
        \tau_i(f_1 \otimes \cdots \otimes f_d)\underline{\tau_i}
        & = (-1)^{s_i s_{i+1}} f_1 \otimes \cdots \otimes f_{i-1} \otimes  (R \circ (f_i \otimes f_{i+1}) \circ R) \otimes f_{i+2} \otimes \cdots f_d \\
        & = (-1)^{s_i s_{i+1}} f_1 \otimes \cdots \otimes f_{i-1} \otimes f_{i+1} \otimes f_i \otimes f_{i+2} \otimes \cdots f_d.
    \end{align*}
    In particular, $\tau_i \cdot e^{\beta}  = e^{\tau_i(\beta)} \cdot \tau_i$. The result then follows.
\end{proof}
Note that since $e^\ell=0$, $e^{\alpha} = 0$ whenever $\alpha$ has a parts $\geq \ell$. 

Until now, no assumption was made on the characteristic of $\field$, but for the rest of this section, we will need $\field \symgroup{d}$ to be semisimple, and hence for $d!$ to be invertible in $\field$. With this assumption, every strict polynomial functors can be written in a certain form, given in the following lemma. 
\begin{lemma}[Schur-Weyl duality]\label{lem:astensor}
    Let $F \in \polyfunctor{d}_1$ be an homogeneous polynomial functor of degree $d$. Assume that $\field$ has characteristic zero or of characteristic $p>d$. Then there exist a left $\field \symgroup{d}$-module $M$ such that $F \cong \bigotimes^d \otimes_{\field \symgroup{d}} M$.
\end{lemma}
A proof of this lemma can be found for the equivalent framework of modules over Schur algebras in \cite[Chapter 6]{green2006polynomial}. The next definition introduce the notation $G^*(E_1 \otimes I)$ used in \ref{thm-intro-1}.

\begin{definition}\label{def:G(V*I)}
    Let $M$ be a left $\field \symgroup{d}$-module $M$. Set $G$ to be the homogeneous strict polynomial functor $G = \bigotimes^d \otimes_{\field \symgroup{d}} M$. Let $V^*$ be a finite dimensional graded $\field$-vector space. Then, for $t \in \integer$, define
    \[
        G^t(V^* \otimes I) = ((V^{\otimes d})^t \otimes \mbox{$\bigotimes^d$}) \otimes _{\field
        \symgroup{d}} M
    \]
    where $\field \symgroup{d}$ acts diagonally on $(V^{\otimes d})^t \otimes \bigotimes^d$ without Koszul signs. This is a homogeneous strict polynomial functor of degree $d$.
\end{definition}

\begin{remark}
    This definition is a special case of the definition in \cite[section 2.5]{touze2012troesch}.
\end{remark}

We can now state and prove our main theorem.
\begin{theorem}\label{theo:Ext(twisted,twisted)}
    Assume that $\field$ is a field of characteristic zero or of characteristic $p>d$. Let $F,G \in \polyfunctor{d}_1$ be homogeneous strict polynomial functors of degree $d$. Then 
    \[
        \Ext^*_{\quantumfunctor{}}(F^{(1)_q},G^{(1)_q}) \cong  \Hom_{\polyfunctor{}_1}(F,G^*(E_1 \otimes I)).
    \]
    where $E_1 = \Ext^*_{\quantumfunctor{}}(I^{(1)_q},I^{(1)_q})$.
\end{theorem}
\begin{proof}
    Under our assumption on $\field$, we can apply lemma \ref{lem:astensor}, and hence there exist left $\field \symgroup{d}$-modules $M,N$ such that
    \[
        F \cong \mbox{$\bigotimes^d$} \otimes_{\field \symgroup{d}} M \quad \mbox{and} \quad G \cong \mbox{$\bigotimes^d$} \otimes_{\field \symgroup{d}} N.
    \]
    Now, since the Schur algebra acts on the $\bigotimes^d$ factors,
    \[
        F^{(1)_q} \cong \mbox{$\bigotimes^{d(1)_q}$} \otimes_{\field \symgroup{d}} M \quad \mbox{and} \quad G^{(1)_q} \cong \mbox{$\bigotimes^{d(1)_q}$} \otimes_{\field \symgroup{d}} N.
    \]
    Moreover, again because of our assumption, the functors $M^* \otimes_{\field \symgroup{d}} -$ and $- \otimes_{\field \symgroup{d}} N$ are exact, and hence :
    \begin{align*}
        \Ext^*_{\quantumfunctor{}}(F^{(1)_q},G^{(1)_q})
        & \cong M^* \otimes_{\field \symgroup{d}} \Ext^*_{\quantumfunctor{}}(\mbox{$\bigotimes^{d(1)_q}$},\mbox{$\bigotimes^{d(1)_q}$}) \otimes_{\field \symgroup{d}} N \\
        & \cong M^* \otimes_{\field \symgroup{d}} \left ( E_1^{\otimes d} \otimes \Hom_{\polyfunctor{}_1}(\mbox{$\bigotimes^d$},\mbox{$\bigotimes^d$}) \right ) \otimes_{\field \symgroup{d}} N \\
        & \cong M^* \otimes_{\field \symgroup{d}} \Hom_{\polyfunctor{}_1}(\mbox{$\bigotimes^d$},\mbox{$\bigotimes^d$}(E_1 \otimes I)) \otimes_{\field \symgroup{d}} N \\
        & \cong \Hom_{\polyfunctor{}_1}(F,G(E_1 \otimes I))
    \end{align*}
\end{proof}
\begin{example}
    Under the hypothesis of theorem \ref{theo:Ext(twisted,twisted)}, we have
    \[
    \Ext^*_{\quantumfunctor{}}(\Lambda^{d(1)_q},\Lambda^{d(1)_q}) \cong S^d(E_1)
    \]
    as a $\field$-vector space. To show it, first remark that $\Lambda^d \cong \field^{sign} \otimes_{\field \symgroup{d}} \bigotimes^d$, where $\field^{sign}$ is the sign representation. Thus,
    \[
        \Ext^*_{\quantumfunctor{}}(\Lambda^{d(1)_q},\Lambda^{d(1)_q}) \cong \field^{sign} \otimes_{\field \symgroup{d}} \left ( E_1^{\otimes d} \otimes \field \symgroup{d} \right ) \otimes_{\field \symgroup{d}} \field^{sign}
    \]
    It is then easy to show that 
    \[
        \field^{sign} \otimes_{\field \symgroup{d}} \left ( E_1^{\otimes d} \otimes \field \symgroup{d} \right ) \otimes_{\field \symgroup{d}} \field^{sign} \cong \field^{triv} \otimes_{\field \symgroup{d}} E_1^{\otimes d} \cong S^d(E_1).
    \]
    For similar reason,
    \begin{align*}
        \Ext^*_{\quantumfunctor{}}(S^{d(1)_q},S^{d(1)_q}) & \cong S^d(E_1),\\
        \Ext^*_{\quantumfunctor{}}(S^{d(1)_q},\Lambda^{d(1)_q}) & \cong \Lambda^d(E_1),\\
        \Ext^*_{\quantumfunctor{}}(\Lambda^{d(1)_q},S^{d(1)_q}) & \cong \Lambda^d(E_1).
    \end{align*}
\end{example}

\section{Computations in small characteristic}\label{sec:Extsmall}

The proof of theorem \ref{theo:Ext(twisted,twisted)} uses in an essential way the semisimplicity of $\field \symgroup{d}$, which is true only under some hypothesis on the characteristic of $\field$. In this section, we make some steps towards conjecture \ref{conj} of the introduction. Let us first start by defining the iterated Frobenius.
\begin{definition}
    If $\field$ is a field of characteristic $p>0$, there is a functor
    \[
        \polyfunctor{d} \to \polyfunctor{dp}, \quad F \mapsto F^{(1)}
    \]
    called the classical Frobenius twist \cite[section 4]{friedlander1997cohomology}. By iterating the Frobenius twist functors, we obtain functors
    \[
        \polyfunctor{d} \to \polyfunctor{dp^r}, \quad F \mapsto F^{(r)}
    \]
    We can then apply the quantum Frobenius twist, and we obtain functors
    \[
        \polyfunctor{d} \to \quantumfunctor{dp^r\ell}, \quad F \mapsto F^{(r+1)_q}
    \]
\end{definition}
\begin{proposition}
If $F \in \polyfunctor{d}_1$, then for all $\alpha \in \Omega(d,n)$
we have 
        $F^{(r)_q}(n)_{p^{r-1}\ell\alpha} = F(n)_\alpha\;.$
\end{proposition}
\begin{proof}
    This follows from the facts that the classical Frobenius twist multiplies the weights by $p$, and the quantum Frobenius twist by $l$, see proposition \ref{property of quantum Frobenius twist}.
\end{proof}

Our first goal is to compute
\[
    E_r = \Ext_{\quantumfunctor{}}^*(I^{(r)_q},I^{(r)_q})
\]
for $r>0$. We have already done the case $r=1$ in theorem \ref{theo:Ext(fro,fro)}. We follow the method of \cite[section 4]{friedlander1997cohomology}. The latter relies on two ingredients : the Koszul and the De Rham complexes, and the Pirashvili vanishing lemma. The quantum version of the Koszul and De Rham complexes have been constructed in section 4, and Pirashvili's lemma is \ref{lem:generalized pirashvili}. Thus, we have already adapted all the necessary homological algebra for this task, and the three following results can be directly obtained by a nearly direct adaptation of the proof of proposition 4.4 to corollary 4.8.
\begin{theorem}\label{Ext}
    Let $\field$ be a field of characteristic $p>0$. Then
    \[
        \Ext_{\quantumfunctor{}}^s(I^{(r)_q},S^{p^{r-j}(j)_q}) = \left \{
        \begin{array}{rl}
             \field & \mbox{if } s \equiv 0 \modulo 2p^{r-j} \mbox{ and } s < 2p^{r-1}\ell \\
             0 & \mbox{otherwise} 
        \end{array}
        \right .
    \]
\end{theorem}

\begin{proof}[Sketch of proof]
This theorem, corresponding to \cite[theorem 4.5]{friedlander1997cohomology}, is proven by induction on $j$. We just give here the main steps. We drop the index `$_{\quantumfunctor{}}$', on $\Ext$ for the sake of readability.

{\bf The base case.} For $j=0$, $\Ext^*(I^{(r)_q},S^{p^{r-1}\ell})$ is easy to compute :  $S^{p^{r-1}\ell}$ is injective and the result thus follows from proposition \ref{property of symmetric power} (remember that the Frobenius twist just multiplies the weights).
\medskip

{\bf The induction step.} We assume the result for $j$.
\begin{itemize}
    \item The acyclicity of the Koszul complex and our quantum Pirashvili lemma \ref{lem:generalized pirashvili} yields isomorphisms
    \[
        \Ext^{*}(I^{(r)_q},S^{p^{r-j}(j)_q}) \cong \Ext^{*+i}(I^{(r)_q},N^{(j)_q}_{p^{r-j}}) \cong \Ext^{*+p^{r-j}-1}(I^{(r)_q},\Lambda^{p^{r-j}(j)_q})
    \]
    (when $j=0$, replace $p^{r-j}$ by $p^{r-1}l$) and hence we know the $\Ext$-group $\Ext^*(I^{(r)_q},N^{(j)_q}_{p^{r-j}})$ and $\Ext^*(I^{(r)_q},\Lambda^{p^{r-j}(j)_q})$.
    \item Using the knowledge of $\Ext^*(I^{(r)_q},\Lambda^{p^{r-j}(j)_q})$, $\Ext(I^{(r)_q},S^{p^{r-j}(j)_q})$, the quantum Pirashvili lemma, and a first hypercohomology spectral sequence, we compute the hyperext groups:
    \[
        \Ext(I^{(r)_q}, \Omega^{\bullet(j)_q}_{p^{r-j}}).
    \]
    (The spectral sequence has only two non zero rows, so the only differentials to study are from $\Ext(I^{(r)_q},\Lambda^{p^{r-j}(j)_q})$ to $\Ext(I^{(r)_q},S^{p^{r-j}(j)_q})$, and these differentials are zero by lacunarity.)
   
    \item A similar calculation shows that the hyperext groups $\Ext(I^{(r)_q},N^{\bullet(j)_q}_{p^{r-j}})$ are zero in odd degrees.
    \item We then consider the second hypercohomology spectral sequences, which by the computation of the homology of both $N$ and $\Omega$, are of the form
    \[
        E_2^{s,t} = \Ext^s(I^{(r)_q},N^{t(j+1)_q}) \quad \mbox{and} \quad \overline{E}_2^{s,t} = \Ext^s(I^{(r)_q},\Omega^{t(j+1)_q})
    \]
    This passage from $j$ to $j+1$ by taking the homology is the central point of our induction. Also note that for the same reason as before, $\overline{E}$ has only two non zero rows.
     \begin{enumerate}
    \item The inclusion $N \hookrightarrow \Omega$ defines a morphism of spectral sequences $E \to \overline{E}$ which in bidegree $s,0$ is an isomorphism since $N^0 = \Omega^0$. We use it to show that all the differentials going to $E^{s,0}$ are zero, and hence $E_\infty^{s,0} = E_2^{s,0} = \Ext(I^{(r)_q},S^{p^{r-j-1}(j+1)_q})$.
    \item The injectivity of the edge morphism $E^{s,0}_\infty \to \Ext(I^{(r)_q},N^{\bullet(j)_q}_{p^{r-j}})$, and the knowledge that $\Ext(I^{(r)_q},N^{\bullet(j)_q}_{p^{r-j}})$ is zero in odd degree then imply that $\Ext(I^{(r)_q},S^{p^{r-j-1}(j+1)_q})$ is zero in odd degree.
    \item Then with the fact that $\overline{E}$ has only two non zero row, we obtain exact sequences, and combining this with the knowledge of $\Ext(I^{(r)_q}, \Omega^{\bullet(j)_q}_{p^{r-j}})$, we obtain the result.
    \end{enumerate}
\end{itemize}
\end{proof}
Moreover, we also have \cite[(4.5.6)]{friedlander1997cohomology}, and hence also \cite[corollary 4.6, 4.7 and 4.8]{friedlander1997cohomology}, which corresponds to the following two propositions.

\begin{proposition}\label{presentation E_r}
    $E_r = \Ext_{\quantumfunctor{}}^*(I^{(r)_q},I^{(r)_q})$ is a graded algebra, generated by elements $e_{r,i}$, $i = 0,...,r-1$, where
    \[
        e_{r,i} \in \Ext^{2p^i}_{\quantumfunctor{}}(I^{(r)_q}, I^{(r)_q})
    \]
    Moreover, we have the following relations
    \[
        \begin{array}{rcll}
             e_{r,i}^p & = & 0 & \mbox{for } i = 0,...,r-2, \\
             e_{r,r-1}^\ell & = & 0, & \\
             e_{r,i} e_{r,j} & = & e_{r,j} e_{r,i} & \mbox{for } i,j = 0,...,r-2, \\
             e_{r,r-1} e_{r,i} & = & \lambda_{r,i} e_{r,i} e_{r,r-1} & \mbox{for } i = 0,...,r-2,
        \end{array}
    \]
    where the $\lambda_{r,i} \in \field$ are non-zero scalars.
\end{proposition}
When $r$ is clear from the context, we omit it in the indices of the generator. Note that in the classical case \cite[theorem 4.10]{friedlander1997cohomology} we know that the algebra $\Ext(I^{(r)},I^{(r)})$ is commutative. In our case, we think that it should also be true, but we were unable to prove it (to adapt the method of Friedlander and Suslin \cite[theorem 4.10]{friedlander1997cohomology}, one would need to twist again extensions between $I^{(r)_q}$ and $I^{(r)_q}$, and we do not know how to do that).

\begin{proposition}\label{freeness of Ext}
    Let $\phi : I^{(r-j)} \to S^{p^{r-j}} \in \Ext^0_{\polyfunctor{}}(I^{(r-j)},S^{p^{r-j}})$ be the natural embedding and let
    \[
        \phi^{(j)_q}  : I^{(r)_q} \to S^{p^{r-j}(j)_q} \in \Ext^0_{\quantumfunctor{}}(I^{(r)_q},S^{p^{r-j}(j)_q})
    \]
    Regard $\Ext^*_{\quantumfunctor{}}(I^{(r)_q},S^{p^{r-j}(j)_q})$ as a right $E_r$-module via the Yoneda product. Then it is a free module of the sub-algebra of $E_r$ generated by $e_{r-j}, e_{r-j+1},..., e_{r-1}$, spanned by $\phi^{(j)_q}$.
\end{proposition}

\begin{proof}[Sketch of proof of propositions \ref{presentation E_r} and \ref{freeness of Ext}]
To define the generators $e_{r,i}$, consider the second hypercohomology spectral sequence $\overline{E}$ corresponding to the complex $\Omega_{p^{r-1}\ell}$ :
\[
    \overline{E}_2^{s,t} = \Ext^s_{\quantumfunctor{}}(I^{(r)_q}, \Omega^{t(1)_q}_{p^{r-1}}) \implies \Ext^{t+s}_{\quantumfunctor{}}(I^{(r)_q}, \Omega_{p^{r-1}\ell}).
\]
By the quantum Pirashvili lemma, this spectral sequence has only two non-zero rows : $t=0$ and $t=p^{r-1}$, and hence has only one non-zero differential $d_{p^{r-1}+1}$, hence $\overline{E}_2 = \overline{E}_{p^{r-1}+1}$. Moreover, 
\[
    d_{p^{r-1}+1} : \Ext^{p^{r-1}-1}_{\quantumfunctor{}}(I^{(r)_q}, \Omega^{t(1)_q}_{p^{r-1}}) \cong \Ext^0_{\quantumfunctor{}}(I^{(r)_q},S^{p^{r-1}(1)_q}) \to \Ext^{2p^{r-1}}_{\quantumfunctor{}}(I^{(r)_q},S^{p^{r-1}(1)_q}).
\]
Then there is an unique element $e_{r,r-1}$ such that $\phi^{(1)_q} \cdot e_{r,r-1}$ (here, we use the Yoneda product). For $j=0,...,r-2$, the generator $e_{r,j}$ is obtained from the generator in degree $2p^j$ defined in \cite[section 4]{friedlander1997cohomology} via the twisting morphisms
\[
    \Ext^{2p^j}_{\polyfunctor{}}(I^{(r-1)},I^{(r-1)}) \to \Ext^{2p^j}_{\quantumfunctor{}}(I^{(r)_q},I^{(r)_q})
\]
which are algebra morphisms. Now, to prove that they are generators of $E_r$, we can prove a similar result to \cite[Corollary 4.6]{friedlander1997cohomology}, and then that the graded subspace of $E_r$ spanned by the products
\[
    (e_{r,0})^{\alpha_0} (e_{r,1})^{\alpha_1} \cdots (e_{r,r-1})^{\alpha_{r-1}} \qquad (*)
\]
with $0 \leq \alpha_0,...,\alpha_{r-2} < p$ and $0 \leq \alpha_{r-1} < \ell$, check the condition corresponding to the condition of \cite[Corollary 4.6]{friedlander1997cohomology}, and hence is equal to $E_r$. More generally, this method can be used to prove proposition \ref{freeness of Ext}. Now, we explain how to prove the relations. For the relation $e_{r,r-1}^\ell = 0$, note that $e_{r,r-1}^\ell \in E^{2p^{r-1}\ell}_r$, which is zero by theorem \ref{Ext}. The two relation $e_{r,i}^p = 0$ and $e_{r,i}e_{r,j} = e_{r,j}e_{r,i}$ are obtained by applying the twisting morphism to the corresponding relations in $\Ext^*_{\polyfunctor{}}(I^{(r-1)},I^{(r-1)})$. For the last relation, note that $e_{r,i} e_{r,r-1} \neq 0$, since it is the only product in $(*)$ of degree $2p^i + 2p^{r-1}$. Since $E_r^{2p^i + 2p^{r-1}}$ is one dimensional by theorem \ref{Ext}, it suffices to prove that $e_{r,r-1} e_{r,i} \neq 0$. To prove it, it is the same method which proves $e_{r,i} e_{r,r-1} \neq 0$, but replacing the symmetric power by their dual, the divided power, and using the dual Koszul and de Rham complex $K^\#$ and $\Omega^\#$.
\end{proof}

The following definitions are used to define the isomorphism of theorem \ref{thm-intro-2}.
\begin{definition}
    The functor $-^{(s)_q} : \polyfunctor{}_1 \to \quantumfunctor{}$ is exact, and hence, for any $F,G \in \polyfunctor{}_1$, it induces a map 
    \[
        \Ext_{\polyfunctor{}_1}^*(F,G) \to \Ext_{\quantumfunctor{}}^*(F^{(s)_q},G^{(s)_q}).
    \]
    The image of an extension $x \in \Ext_{\polyfunctor{}_1}^*(F,G)$ is denoted by $x^{(s)_q}$.
\end{definition}
\begin{definition}
    If $V = \bigoplus V^i$ is a graded vector space, then $V^{(r)}$ denotes the graded vector space
\[
    (V^{(r)})^k = \left \{
    \begin{array}{rl}
         V^i & \mbox{if } k = p^{r-1} i, \\
         0 & \mbox{otherwise.} 
    \end{array}
    \right .
\]
\end{definition}

We can now prove theorem \ref{thm-intro-2}. Its proof is a direct adaptation of the proof of Giordano \cite{giordano2023additive}. The only difference in his case is that it is known that the map corresponding to $\sigma$ below is, in fact, an algebra map. In our case, we do not know if it is true (we certainly think so), but we do not need it.
\begin{theorem}\label{final theorem}
    Let $\sigma : (\Ext^*_{\quantumfunctor{}}(I^{(s)_q},I^{(s)_q}))^{(r)} \to \Ext^*_{\quantumfunctor{}}(I^{(r+s)_q},I^{(r+s)_q})$ be the graded injective map defined by
    \[
        e_{s,0}^{\alpha_0} e_{s,1}^{\alpha_1} \cdots e_{s,s-1}^{\alpha_{s-1}} \mapsto e_{r+s,r}^{\alpha_0} e_{r+s,r+1}^{\alpha_1} \cdots e_{r+s,r+s-1}^{\alpha_{s-1}}
    \]
    Let $G \in \polyfunctor{p^r}$ be a strict polynomial functor. 

    \noindent
    Regard $\Ext^*_{\quantumfunctor{}}(I^{(r+s)_q},G^{(s)_q})$ as a right $\Ext^*_{\quantumfunctor{}}(I^{(r+s)_q},I^{(r+s)_q})$-module via the Yoneda product, and define the graded map
    \[
        \begin{array}{rcl}
             \Psi : \Ext^*_{\polyfunctor{}_1}(I^{(r)},G) \otimes (\Ext^*_{\quantumfunctor{}}(I^{(s)_q},I^{(s)_q}))^{(r)} & \to & \Ext^*_{\quantumfunctor{}}(I^{(r+s)_q},G^{(s)_q}) \\
             x \otimes e & \mapsto & x^{(s)_q} \cdot \sigma(e) 
        \end{array}
    \]
    Then $\Psi$ is an isomorphism.
\end{theorem}
This follows from a more general fact :
\begin{proposition}\label{prop:formalisation}
    Let $\field$ be a field, let $\mathcal{C}_1, \mathcal{C}_2$ be two $\field$-linear abelian categories with enough projectives and injectives, and let $F : \mathcal{C}_1 \to \mathcal{C}_2$ be an exact functor (thus $F$ induces a morphism on $\Ext$-groups). Let $x \in \mathcal{C}_1$. Suppose we are given a graded space $E$ and a graded linear map $\sigma : E \to \Ext_{\mathcal{C}_2}(F(x),F(x))$ such that the map
    \[
        \begin{array}{rcl}
             \Psi_y : \Ext_{\mathcal{C}_1}(x,y) \otimes E & \to & \Ext_{\mathcal{C}_2}(F(x),F(y)) \\
             f \otimes a &  \mapsto & F(f) \cdot \sigma(a)
        \end{array}
    \]
    is a graded isomorphism whenever $y$ is injective. Then this is a graded isomorphism for any $y \in \mathcal{C}_1$.
\end{proposition}
\begin{proof}
    The proof of \cite[proposition 4.6]{giordano2023additive} works the same in our more general context. The hypothesis that $\sigma$ is an algebra morphism is not used.
\end{proof}

\begin{proof}[Proof of theorem \ref{final theorem}]
    By proposition \ref{prop:formalisation}, we only need to prove theorem \ref{final theorem} when $G$ is injective.
    It is enough to consider the case $G = S^{\mu}$. If $\mu \neq (p^r)$, then both the source and target of $\Psi$ are zero, by Pirashvili's vanishing lemma \ref{lem:generalized pirashvili}, hence there is nothing to prove. Otherwise, if $\mu = (p^r)$, then this follows directly from proposition \ref{freeness of Ext}.
\end{proof}
We end this paper with some examples of concrete computations in small characteristic.
\begin{example}
    We can find back theorem \ref{Ext} with this formula. For all $r\ge 1$ and all $s\ge 0$ there is a graded isomorphism
    $$\Ext^k_{\quantumfunctor{}}(I^{(r+s)_q},S^{p^{r}\,(s)_q})\simeq 
    \begin{cases}
    \field & \text{if $k=2p^r i$ for $0\le i<p^{s-1}\ell$,}\\
    0 & \text{otherwise.}
    \end{cases}
    $$
    This follows directly from 
    \[
        \Ext^k_{\quantumfunctor{}}(I^{(r+s)_q},S^{p^{r}\,(s)_q}) \simeq \bigoplus_{i + p^r j = k} \Ext^i_{\polyfunctor{}_1}(I^{(r)},S^{p^r}) \otimes E_s^{j} \simeq \bigoplus_{p^r j = k} E^j_s.
    \]
    The last isomorphism comes from the fact that $\Ext^i_{\polyfunctor{}_1}(I^{(r)},S^{p^r})$ is equal to $0$ if $i \neq 0$, and equal to $\field$ if $i=0$. In the same fashion,
    $$\Ext^k_{\quantumfunctor{}}(I^{(r+s)_q},\Lambda^{p^{r}\,(s)_q})\simeq 
    \begin{cases}
    \field & \text{if $k=2p^r i + p^r-1$ for $0\le i<p^{s-1}\ell$,}\\
    0 & \text{otherwise.}
    \end{cases}
    $$
\end{example}

\begin{remark}
    In characteristic $p>0$, the following conjecture is a consequence of conjecture \ref{conj}.
    \begin{conjecture}\label{conj pos}
        For all strict polynomial functors $F,G$, there is a graded isomorphism
        \[
        \Ext^k_{\quantumfunctor{}}(F^{(s)_q},G^{(s)_q}) \simeq \bigoplus_{i+j=k} \Ext^i(F,G^j(E_s \otimes I)).
        \]
        where $E_s = \Ext_{\quantumfunctor{}}^*(I^{(s)_q},I^{(s)_q})$.
    \end{conjecture}
    This follows from conjecture \ref{conj} using the analogous result for the classical case (see \cite[Theorem 4.6]{touze2018cohomology}).
\end{remark}

\begin{remark}\label{rmk:other parametrization}
    For any strict polynomial functor $F$, and any graded space $V^*$, we can define graded strict polynomial functors $F^V_j = F^j(\Hom_{\field}(V^*,I))$ for $j\geq 0$, in a similar way as in definition \ref{def:G(V*I)} and \cite[section 2.5]{touze2012troesch}. Then, there is a bigraded isomorphism:
    \[
        \Ext^i(F,G^j(V^* \otimes I)) \cong \Ext^i(F^V_j,G).
    \]
    (see \cite[proposition 7.3]{touze2012troesch}). This can be more practical to use $F^V_j$ than $G^j(V^* \otimes I)$ for some computation.
\end{remark}

\begin{remark}
    The conjecture \ref{conj pos} is verified in the case where $F$ is of the form $I^{(r)_q}$, as a consequence of theorem \ref{final theorem}. To see why, we use remark \ref{rmk:other parametrization}. In the particular case $F = I^{(r)}$ and with this remark, conjecture \ref{conj pos} asks for the existence of a graded isomorphism
    \[
        \Ext^k_{\quantumfunctor{}}(I^{(r+s)_q},G^{(s)_q}) \cong \bigoplus_{i+j=k}  \Ext^i_{\polyfunctor{}_1}((I^{(r)})^{E_s}_j, G).
    \]
    Since 
    \[
    (I^{(r)})^{E_s}_j \cong \left \{ \begin{array}{cl}
        I^{(r)} & \mbox{if } 2p^r \mbox{ divides } j \mbox{ and } j < 2p^{r+s-1}l, \\
        0 & \mbox{otherwise},
    \end{array} \right .
    \]
    the conjecture asks the existence of a graded isomorphism
    \[
        \Ext^k_{\quantumfunctor{}}(I^{(r+s)_q},G^{(s)_q}) \cong \bigoplus_{\substack{i+2p^rj=k \\ 0 \leq j < p^{s-1}l}} \Ext^i_{\polyfunctor{}_1}(I^{(r)}, G) \cong \Ext^*_{\polyfunctor{}_1}(I^{(r)},G) \otimes (\Ext^*_{\quantumfunctor{}}(I^{(s)_q},I^{(s)_q}))^{(r)}.
    \]
    Such an isomorphism is given by theorem \ref{final theorem}.
\end{remark}

\bibliography{biblio.bib}
\bibliographystyle{plain}

\end{document}